\documentclass[times,sort&compress,3p]{elsarticle}
\journal{Journal of Multivariate Analysis}
\usepackage[labelfont=bf]{caption}

\usepackage{amsmath,amsfonts,amssymb,amsthm,booktabs,color,epsfig,graphicx,hyperref,url}

\usepackage{bm}

\usepackage{diagbox}

\usepackage[noend]{algpseudocode}

\usepackage{algorithmicx,algorithm}

\usepackage{mathrsfs}
\usepackage{multirow}
\usepackage{threeparttable}
\usepackage{cleveref}
\RequirePackage{subfigure}

\theoremstyle{plain}
\newtheorem{theorem}{Theorem}

\newtheorem{rem}{Remark}
\newtheorem{proposition}{Proposition}
\newtheorem{lemma}{Lemma}

\theoremstyle{definition}
\newtheorem{definition}{Definition}

\newtheorem{example}{Example}
\newtheorem{Assum}{Assumption}
\begin{document}

\begin{frontmatter}

\title{Distributed estimation of spiked eigenvalues in spiked population models}

\author[1]{Lu Yan }
\author[1]{Jiang Hu\corref{mycorrespondingauthor}}

\address[1]{KLASMOE and School of Mathematics $\&$ Statistics, Northeast Normal University, China}

\cortext[mycorrespondingauthor]{Corresponding author. Email addresses: \url{yanl129@nenu.edu.cn} (L. Yan), \url{huj156@nenu.edu.cn} (J. Hu)}

\begin{abstract}
The proliferation of science and technology has led to the prevalence of voluminous data sets that are distributed across multiple machines. It is an established fact that conventional statistical methodologies may be unfeasible in the analysis of such massive data sets due to prohibitively long computing durations, memory constraints, communication overheads, and confidentiality considerations.
In this paper, we propose distributed estimators of the spiked eigenvalues in spiked population models.  The consistency and asymptotic normality of the distributed estimators are derived, and the statistical error analysis of the distributed estimators is provided as well.   Compared to the estimation from the full sample, the proposed distributed estimation shares the same order of convergence. Simulation study and real data analysis indicate that the proposed distributed estimation and testing procedures have excellent properties in terms of estimation accuracy and stability as well as transmission efficiency.
\end{abstract}


\end{frontmatter}

\section{Introduction\label{sec:1}}

The sample covariance matrix is one of the fundamental statistics in multivariate statistical analysis. Under the classical assumption, where the dimension is fixed and the sample size tends to infinity, the sample covariance matrix serves as a consistent estimator of the population covariance matrix. However, when the dimensionality of the sample approaches the same order of magnitude as the sample size, or even surpasses it, the accuracy of the estimates is not guaranteed. The theory of large-dimensional random matrices becomes crucial in the study of high-dimensional sample covariances. The examination of covariance eigenvalues based on random matrices can be traced back to the Marcenko-Pastur (M-P) law first proposed by \cite{marvcenko1967distribution}.

Specifically, let $ \bm{X} =\left(\bm{x}_{1}, \dots, \bm{x}_n\right)=(x_{ij})_{p\times n}$  and $\bm{S}_n=\frac{1}{n}\bm{V}\bm{X}\bm{X}^*\bm{V}^*$, where $\{x_{ij},1\leq i\leq p,1\leq j\leq n\}$ are  independent and identically distributed (i.i.d.) random variables with mean zero and variance one. The empirical spectral distribution (ESD) of the random matrix $\bm{S}_n$ is defined by
$$F_n(x)=p^{-1}\sum_{i=1}^{p}\delta(\lambda_i(\bm{S}_n)\le x),$$
where $\lambda_1(\bm{S}_n)\ge \cdots \ge\lambda_p (\bm{S}_n)  $  are the eigenvalues of $\bm{S}_n$, $\delta(\cdot)$ is an indicator function.   
It is proved that if  $\bm{V}=\bm{I}_p$ is the identity matrix,  as $n\to \infty$ and $p/n \to y\in (0,\infty)$, with probability one, the ESD of $\bm{S}_n$ converges weakly to the M-P law whose density function is given by 
\begin{equation*}
	\frac{d}{dx}F_y(x)=  
	\left\{
	\begin{array}{ll}
		\frac{1}{2\pi x y \sigma^2}\sqrt{(b-x)(x-a)},\quad & \text{if} \quad  a \le x \le b,\\
		0,\quad	&\text{otherwise},\\
	\end{array}
	\right.
\end{equation*}
and has a point mass $1-1/y$ at the origin if $y >1$, where $a=\sigma^2(1-\sqrt{y})^2$ and $b=\sigma^2(1+\sqrt{y})^2$. If the fourth moment of $x_{11}$ exists, then with probability one
$$\lim_{n\to \infty}\lambda_{1}(\bm{S}_n)=b~\mbox{and}~~ \lim_{n\to \infty}\lambda_{p}(\bm{S}_n)=a.$$
When $\bm{V}$ is the identity matrix but has only a few eigenvalues that are not equal to one, we refer to this as the spiked population model, which was coined by \cite{johnstone2001distribution}. Research on extreme eigenvalues initially emerged in \cite{geman1980limit}. They established that, subject to specific moment conditions, the dominant eigenvalue of the sample covariance matrix converges to $b$ when $p/n \to y \in (0,\infty)$. Furthermore, \cite{yin1988limit} extends the conclusion to encompass the condition of finite fourth-order moments. In \cite{bai2010spectral}, the central limit theorem (CLT) for these spiked eigenvalues is presented, and in \cite{bai2012estimation}, the results are extended to encompass general spiked models. In \cite{baik2006eigenvalues}, the limit of convergence for the eigenvalues of the sample covariance matrix in a general family of samples is completely determined. In recent years, the study of spiked eigenvalues remains a prominent area of research. In \cite{cai2020limiting}, the authors investigated the asymptotic distributions of both the spiked eigenvalues and the maximum nonspiked eigenvalues of the sample covariance matrices under a general covariance model with diverging spiked eigenvalues. In \cite{zhang2022asymptotic}, the authors demonstrate, under a general high-dimensional spiked sample covariance model, that the leading sample spiked eigenvalues and their linear spectral statistics are asymptotically independent when the sample size scales proportionally with dimension. Furthermore, the need for the block diagonal assumption on the overall covariance matrix, which is typically required in the literature, is alleviated. Similarly, in \cite{jiang2021limits}, for a class of generalized spiked Fisher matrices, the authors also dispense with the assumption that the covariance matrix is diagonal or diagonal block-structured. They proceed to investigate the almost sure limits of the sample spiked eigenvalues when the population covariance matrix is arbitrary. Furthermore, \cite{hou2023spiked} explored the asymptotic behavior of the spiked eigenvalues of non-central Fisher matrices and derived the CLT for these eigenvalues.

With the continuous evolution of information technology, the significance of handling vast quantities of data has grown substantially in recent years. Due to challenges such as prolonged computing times, memory limitations, communication overheads, and confidentiality considerations, the practice of distributed data analysis has gained paramount importance. Many classical statistical estimators have adapted to distributed frameworks.
In the domain of linear regression, \cite{dobriban2021distributed} introduced the concept of distributed least squares weighted estimation. When the dimensionality exceeds the sample size, rendering standard least squares estimation impractical, \cite{battey2018distributed} proposed distributed estimation techniques for sparse parameter vectors, employing the debiased estimator of the lasso. Additionally, \cite{chen2020distributed} addressed the challenge of testing under conditional assumptions in the context of distributed estimation, specifically for sparse parameters and heavy-tailed noise.
Within Principal Component Analysis (PCA), \cite{fan2019distributed} advocated an averaging approach for estimating the principal eigenspace within a distributed framework. Notably, \cite{li2021robust} extended their algorithm to accommodate cases involving heavy-tailed distributions, in contrast to the assumption of sub-Gaussian distribution in \cite{fan2019distributed}.
Furthermore, \cite{duan2022heterogeneity} introduced a novel likelihood estimation method rooted in likelihood estimation, applying it to a distributed heterogeneous environment. \cite{wang2019distributed} provided fresh insights into the statistical properties of support vector machines, proposing a multi-round linear estimator tailored for distributed environments. The integration of a distributed framework with a bootstrap approach to enhance computational efficiency was presented in \cite{chen2021distributed}. Moreover, a distributed architecture for the two-sample U statistic was advocated by \cite{huang2023distributed}, with a notable focus on privacy concerns, a topic also addressed by \cite{imtiaz2018differentially}. These represent current and highly pertinent issues in the field.

This paper investigates the distributed estimation of spiked eigenvalues in the spiked population model. We introduce a weighted estimator, which offers the significant advantage of requiring only one communication step, leading to higher transmission efficiency compared to direct averaging. Additionally, weighting proves to be an effective approach for handling heterogeneity.
In \cite{dobriban2021distributed}, we observe the first instance of distributed estimation in the weighted case, where the expression for the weights is straightforward due to the least squares estimation. However, in \cite{dobriban2020wonder}, as ridge estimation is not unbiased, their weights need not necessarily sum to 1. Furthermore, the weights in \cite{dobriban2020wonder} involve the estimation of unknown parameters, aligning with our approach to weight assignment in this paper.
In \cite{gu2022weighted}, emphasis is placed on the M-estimation problem in the context of heterogeneity. The paper proposes a weighting matrix to address this issue. While the primary focus of \cite{duan2022heterogeneity} lies in distributed likelihood estimation, it also introduces the concept of weighting to more effectively manage heterogeneity.

The rest of the article is shown below. 
Section \ref{sec:2} focuses on the spiked population model and the CLT for spiked eigenvalues. In Section \ref{sec:3}, under the distributed framework, we design a distributed weighted estimator of spiked eigenvalues, and give the compatibility and asymptotic normality of the weighted estimator. In Section \ref{sec:4}, the theoretical result on the statistical error of the weighted estimator is obtained. In Section \ref{sec:5}, we simulate the asymptotic properties of the weighted estimator as well as the statistical error results. In addition, we also include an empirical analysis of a real case such as rice data, and the results show that our weighted estimator also shows its good properties in real applications. The discussion and future perspectives are presented in Section \ref{sec:6}.

\section{Problem setup\label{sec:2}}

Recall the sample covariance matrix
$$\bm{S}_n=\frac{1}{n}\bm{V}\bm{X}\bm{X}^*\bm{V}^*.$$
where  $\bm{X} =\left(\bm{x}_{1}, \dots, \bm{x}_n\right) =\left(x_{ij}\right)_{p\times n}$, $\bm{V}$ is a $p\times p$ deterministic matrix and define $\bm{\Sigma}=\bm{V}\bm{V}^*$.   Let
$$\lambda_{n,1}\ge \lambda_{n,2}\ge \cdots \ge \lambda_{n,p} $$
be the ordered eigenvalues of the sample covariance matrix $\bm{S}_n$, and
$$\bm{V}=\bm{L}\left(\begin{array}{cc} \bm{A}_M^{1/2} & 0\\
	0 & \bm{I}_{p-M} \end{array}\right)\bm{U}^*,$$
where $\bm{L}$ and $\bm{U}$ are unitary matrices, $\bm{A}_M$=diag$\left(\alpha_1, \alpha_2, \dots, \alpha_M\right)$ is a diagonal matrix, $M$ is fixed. Here, $\alpha_1, \alpha_2, \dots, \alpha_M$ are non-null and non-unit eigenvalues, and we call these eigenvalues  $\alpha_j$  the \emph{spiked population eigenvalues}. 
Chunking $\bm{U}$ as $\bm{U}=\left(\bm{U}_1, \bm{U}_2\right)$, where $\bm{U}_1$ is a $p\times M$ submatrix of $\bm{U}$. Let $\bm{u}_k=\left(u_{k1},u_{k2},\dots,u_{kp}\right)^*$ be the $k$th column of $\bm{U}_1$.

The purpose of this article is to estimate the spiked population eigenvalues that lie outside the interval of $\left[1-\sqrt{y},1+\sqrt{y}\right]$. Suppose that $\bm{A}_M$ has $M_a$  eigenvalues less than $1-\sqrt{y}$ and $M_b$  eigenvalues greater than $1+\sqrt{y}$, i.e.,
$$\alpha_1>\alpha_2>\cdots >\alpha_{M_b}>1+\sqrt{y},$$ 
$$\alpha_M<\alpha_{M-1}<\cdots <\alpha_{M-M_a+1}<1-\sqrt{y}.$$

According to \cite{bai2008methodologies}, it is known that the empirical spectral distribution (ESD) of $\bm{S}_n$ converges weakly to the M-P law with $\sigma^2=1$. Additionally, when $\alpha_k > 1+\sqrt{y}$, \cite{baik2006eigenvalues} demonstrated that:
$$\lambda_{n,k}\to \phi(\alpha_k)=\alpha_k+\frac{y\alpha_k}{\alpha_k-1}, \quad almost \ surely.$$ 
Furthermore, in the work of \cite{bai2008central}, they established the CLT for these extreme sample eigenvalues lying outside the interval $\left[(1-\sqrt{y})^2, (1+\sqrt{y})^2\right]$, based on the following assumptions.

\begin{Assum}
	\label{ass1}
	As $n\to\infty $, $p/n\to y\in (0,1)$.
\end{Assum}
\begin{Assum}
	\label{ass2}
	$\{x_{ij},1\leq i\leq p,1\leq j\leq n\}$ are i.i.d. with $\mathbb{E}x_{ij}=0$, $\mathbb{E}\vert x_{ij}\vert^2=1$, and $\mathbb{E}\vert x_{ij}\vert^4=\gamma_4<\infty $.	
\end{Assum}
\begin{Assum}
	\label{ass3}
	$\alpha_k \notin \left[1-\sqrt{y}, 1+\sqrt{y}\right]$, $k\in\{1,\dots,M_b,M-M_a+1,\dots,M\}$. 
\end{Assum}
\begin{Assum}
	\label{ass4}
	As $n\to \infty$, the ESD ${H}_n$ of $\bm{\Sigma}$ tends to a probability distribution ${H}$. 
\end{Assum}

\begin{lemma}[Theorem 3.1 in \cite{bai2008central}]
	\label{lem0}
	Given that Assumptions \ref{ass1}--\ref{ass4} hold. For each spiked eigenvalue $\alpha_k \notin \left[1-\sqrt{y}, 1+\sqrt{y}\right]$, the random variable 
	$$\sqrt{n}(\lambda_{n,j}-\phi(\alpha_k)),$$
    converges weakly to the Gaussian distribution with mean 0 and variance 
    $$\left(\gamma_4-3\right)\frac{\alpha_k^2\left[\left(\alpha_k-1\right)^2-y\right]^2}{\left(\alpha_k-1\right)^4}\sum_{t=1}^{p}u_{kt}^4+\frac{2\alpha_k^2\left[\left(\alpha_k-1\right)^2-y\right]}{\left(\alpha_k-1\right)^2},$$
  	when $k\in\left\{1,\dots, M_b\right\}$, $j=k$ and $k \in \left\{M-M_a+1,\dots, M\right\}$, $j=p-M+k$. 	
	
\end{lemma}

\begin{rem}
Observing the variance form above, two simple facts emerge: 
$\gamma_4\ge1$ due to the \emph{H$\ddot{o}$lder's inequality} and 
$\sum_{t=1}^{p}u_{kt}^4\le 1$. Therefore, the limit of the variance exists as $p\to\infty$.
\end{rem}

Given the smooth invertibility of the function $\phi(\alpha_k)$, we can employ its inverse to obtain an estimate of $\alpha_k$: 
$$\hat{\alpha}_{n,j}=\frac{\lambda_{n,j}+1-y\pm \sqrt{(\lambda_{n,j}+1-y)^2-4\lambda_{n,j}}}{2},$$
for $\lambda_{n,j}\notin \left[(1-\sqrt{y})^2, (1+\sqrt{y})^2\right] $. 

Regarding the choice of positive and negative signs, for $\alpha_k>1+\sqrt{y}$, it is known that when $y\to 0$ , there should be $\lambda_{n,k}\to \alpha_k$ and $\hat{\alpha}_{n,k}\to \alpha_k$. Also consider that $\phi\left(\alpha_k\right)$ is continuous with respect to $y$, so for $\alpha_k>1+\sqrt{y}$, this estimate should be:
$$\hat{\alpha}_{n,j}=\frac{\lambda_{n,j}+1-y+\sqrt{(\lambda_{n,j}+1-y)^2-4\lambda_{n,j}}}{2}, \quad j=k.$$
Applying the Delta method and Lemma \ref{lem0}, we can derive the asymptotic distribution of $\sqrt{n}\left(\hat{\alpha}_{n,k}-\alpha_k\right)$. Similarly, for $\alpha_k<1-\sqrt{y}$ , we estimate $\alpha_k$ with  
$$\hat{\alpha}_{n,j}=\frac{\lambda_{n,j}+1-y-\sqrt{(\lambda_{n,j}+1-y)^2-4\lambda_{n,j}}}{2}, \quad j=p-M+k,$$
and the asymptotic distribution of $\sqrt{n}\left(\hat{\alpha}_{n,j}-\alpha_k\right)$, can be obtained by Lemma \ref{lem0} and Delta method. Thus, we have the following theorem.
\begin{theorem}
	\label{th5}
	Given that Assumptions \ref{ass1}--\ref{ass4} hold. For each spiked eigenvalue $\alpha_k \notin \left[1-\sqrt{y}, 1+\sqrt{y}\right]$, the random variable 
$$\sqrt{n}(\hat{\alpha}_{n,j}-\alpha_k),$$
converges weakly to the Gaussian distribution with mean 0 and variance 
$$\sigma_{n,j}^2:=\left(\gamma_4-3\right)\alpha_k^2\sum_{t=1}^{p}u_{kt}^4+\frac{2\alpha_k^2\left(\alpha_k-1\right)^2}{\left(\alpha_k-1\right)^2-y},$$
 when $k\in\left\{1,\dots, M_b\right\}$, $j=k$ and $k \in \left\{M-M_a+1,\dots, M\right\}$, $j=p-M+k$. 
\end{theorem}
\begin{rem}
In $\sigma_{n,j}^2$, it is evident that $\gamma_4$ and $\sum_{t=1}^{p}u_{kt}^4$ are unknown. If we intend to utilize the theorem for hypothesis testing related to spiked eigenvalues, it is imperative to estimate both $\gamma_4$ and $\sum_{t=1}^{p}u_{kt}^4$. Lemma \ref{lem3} below provides an estimate for $\sum_{t=1}^{p}u_{kt}^4$. 
For the estimation of $\gamma_4$, refer to Theorem 2.7 in \cite{Zhangh19I} for detailed information.
\end{rem}

\begin{lemma}[Theorem 2.6 in \cite{zhang2022asymptotic} ] 
	\label{lem3}
	Given that Assumption \ref{1}--\ref{4} hold and $\bm{V}$ is symmetric. Let $\hat{\bm{u}}_k=\left(\hat{u}_{k1}, \dots, \hat{u}_{kp}\right)$ be eigenvectors of $\bm{S}_n$ associated with eigenvalue $\lambda_{n,k}$ and $\hat{u}_{kt}$ be the $t$th coordinate of $\hat{\bm{u}}_k$. For $1\le k\le M$, $\sum_{t=1}^{p}u_{kt}^4$ is consistently estimated by $\sum_{j=1}^{p}\left\{\sum_{t=1}^{p}\theta_k\left(t\right)\hat{u}_{tj}^2\right\}^2$, where

\begin{equation}
	\theta_k\left(t\right)=
	\begin{cases}
		-\varphi_k\left(t\right)& \text{ $ t=k, $ } \\
		1+\varrho_k\left(t\right)& \text{ $ t\ne k, $ }
	\end{cases}
\end{equation}
   $$\varphi_k\left(t\right)=\frac{\lambda_{n,k}}{\lambda_{n,t}-\lambda_{n,k}}-\frac{v_i}{\lambda_{n,t}-v_i},$$
   $$\varrho_k\left(t\right)=\sum_{j\ne k }^{p}\left(\frac{\lambda_{n,j}}{\lambda_{n,t}-\lambda_{n,j}}-\frac{v_j}{\lambda_{n,t}-v_j}\right),$$
	and where $v_1\ge v_2\ge \cdots \ge v_p$ are the real valued solutions to the equation in x:
	$$\frac{1}{p}\sum_{k=1}^{p}\frac{\lambda_{n,k}}{\lambda_{n,k}-x}=\frac{1}{y}.$$
	In the expressions of $\varphi_k\left(t\right)$ and $\varrho_k\left(t\right)$, we use the convention that any term of form $\frac{0}{0}$ is 0.
\end{lemma}

In this section, we provided a comprehensive explanation of the estimation of spiked eigenvalues that fall outside the interval $\left[1-\sqrt{y}, 1+\sqrt{y}\right]$, along with its corresponding asymptotic distribution.
The following section will address the approach to estimating these spiked eigenvalues within the framework of a distributed environment.

\section{Weighted average estimator\label{sec:3}}

Suppose the samples are distributed across $m$ machines, with $n_{\ell}$ samples on the $\ell$th machine. These samples are independently and i.i.d., and the sample covariance matrix on the $\ell$th machine is denoted as:
\begin{equation}
	\bm{S}_{n_{\ell}}^{(\ell)}=\frac{1}{n_{\ell}}\bm{V}\bm{X}^{(\ell)}\bm{X}^{(\ell)*}\bm{V}^*,\quad \ell=1, \dots, m.
	\label{1}
\end{equation}
Where  $\bm{X}^{(\ell)} = \left(\bm{x}_{1}^{(\ell)}, \dots, \bm{x}_{n_{\ell}}^{(\ell)}\right) =\left(x_{ij}^{(\ell)}\right)_{p\times {n_{\ell}}}$. 

The overarching question is: How can we estimate the unknown spiked eigenvalues (denoted as $\alpha_k\notin\left[1-\sqrt{y}, 1+\sqrt{y}\right]$) if we aim to perform the majority of computations locally?

We begin by exploring the aggregation of local spiked eigenvalue estimators at a parameter server through one-step weighted averaging. An asymptotically unbiased estimate of $\alpha_{k}$ is obtainable for each machine. Given the independence of the machines, we contemplate a weighted asymptotically unbiased estimate:

\begin{equation}
	\hat{\alpha}_k=\sum_{{\ell}=1}^{m}\omega_{\ell}\hat{\alpha}_{n_{\ell},j}^{(\ell)},
	\label{2}
\end{equation}
with $\sum_{{\ell}=1}^{m}\omega_{\ell}=1$. 
Aach machine provides an estimate of $\alpha_k$ denoted as $\hat{\alpha}_{n_{\ell},j}^{(\ell)}$. On the $\ell$th machine, $\hat{\alpha}_{n_{\ell},j}^{(\ell)}$ is calculated as:
\begin{equation}
	\hat{\alpha}_{n_{\ell},j}^{(\ell)}=\frac{\lambda_{n_{\ell},j}^{(\ell)}+1-y_{\ell}+ \sqrt{(\lambda_{n_{\ell},j}^{(\ell)}+1-y_{\ell})^2-4\lambda_{n_{\ell},j}^{(\ell)}}}{2}
	\label{5}
\end{equation}
for $k\in\left\{1,\dots,M_b\right\}$ and $j=k$.  
Additionally,
\begin{equation}
	\hat{\alpha}_{n_{\ell},j}^{(\ell)}=\frac{\lambda_{n_{\ell},j}^{(\ell)}+1-y_{\ell}- \sqrt{(\lambda_{n_{\ell},j}^{(\ell)}+1-y_{\ell})^2-4\lambda_{n_{\ell},j}^{(\ell)}}}{2}
	\label{7}
\end{equation}
for $k\in\left\{M-M_a+1,\dots,M\right\}$ and $j=p-M+k$, where $\lambda_{n_{\ell},j}^{(\ell)}$ is the eigenvalues of the sample covariance matrix $\bm{S}_{n_{\ell}}^{(\ell)}$ and $y_{\ell}$ is defined by Assumption \ref{ass6} below. The samples on each machine obey the following assumptions.

\begin{Assum}
	\label{ass5}
	$\bm{X}^{(\ell)}$, $\ell=1,\dots,m$,  are independent with $x_{ij}^{(\ell)}$ satisfying Assumption \ref{ass2}.
\end{Assum}

\begin{Assum}
	\label{ass6}
	On each machine, as $n_{\ell}\rightarrow\infty$ and $p/n_{\ell}\to y_{\ell}\in(0,1)$, ${\ell}=1, \dots, m$, $n=\sum_{\ell=1}^{m}n_{\ell}$.
\end{Assum}

\begin{Assum}
	\label{ass7}
	$\alpha_k \notin \left[1-\sqrt{y_{\ell}}, 1+\sqrt{y_{\ell}}\right]$ for all $\ell=1,\dots,m$.
\end{Assum}

	Consider the distributed spiked eigenvalues problem described above and a data set that consists of $n$ data samples. The data set is distributed across $m$ sites. We compute the local spiked eigenvalues estimator $\hat{\alpha}_{n_{\ell},j}^{(\ell)}$ on each data set. Then send the local estimates to the central location, and combine them by weighting, i.e., $\hat{\alpha}_k=\sum_{\ell=1}^{m}\omega_{\ell}\hat{\alpha}_{n_{\ell},j}^{(\ell)}$. Then  we have the following proposition. 
\begin{proposition}[Asymptotically optimal weights]
	\label{pro1}
	Given that Assumption \ref{ass5}--\ref{ass7} hold. The asymptotically optimal weights, which minimize the mean square error of the distributed estimator under the spiked population model as $n_{\ell}\to \infty$, $\ell=1,\dots,m,$ are given by:
  \begin{equation}
  	\omega_{\ell}=\frac{n_{i}/\sigma_{\ell}^2}{\sum_{i=1}^{m}n_{\ell}/\sigma_i^2}, \quad \ell=1, \dots, m,
  	\label{eq1}
  \end{equation}
		where
		\begin{equation}
		\sigma_{\ell}^2:=\left(\gamma_4^{(\ell)}-3\right)\alpha_k^2\sum_{t=1}^{p}{u^{(\ell)}}_{kt}^4+\frac{2\alpha_k^2\left(\alpha_k-1\right)^2}{\left(\alpha_k-1\right)^2-y_{\ell}}, \quad \ell=1, \dots, m.
		\label{eq2}
		\end{equation}
	
	The limit of the mean square error of the asymptotically optimal weighted distributed spiked eigenvalues estimator $\hat{\alpha}_k$ with $m$ nodes is equal to
	$$\lim_{n_{\ell}\to\infty}\mathbb{E}\left(\sum_{\ell=1}^{m}\omega_{\ell}\hat{\alpha}_{n_{\ell},j}^{(\ell)}-\alpha_k\right)^2=\frac{1}{\sum_{\ell=1}^{m}n_{\ell}/\sigma_{\ell}^2}.$$

\end{proposition}

Observing from \cref{eq1} and \cref{eq2}, it is evident that the asymptotically optimal weights incorporate the unknown quantity $\alpha_k$. Consequently, we require an initial estimator to provide an initial estimate, assuming the following conditions.

\begin{Assum}
	\label{ass8}
	The initial value $\bar{\alpha}_k$ satisfies:  $ \bar{\alpha}_k\stackrel{\mathscr{P}}{\longrightarrow}\alpha_k$, as $n_{\ell}\to\infty$, $\ell=1,\dots,m$. 
\end{Assum}

\begin{rem}
	$\bar{\alpha}_k=h\left(\hat{\alpha}_{n_{1},j}^{(1)}, \hat{\alpha}_{n_{2},j}^{(2)}, \dots, \hat{\alpha}_{n_{m},j}^{(m)}\right) $, where $h\left(\hat{\alpha}_{n_{1},j}^{(1)}, \hat{\alpha}_{n_{2},j}^{(2)}, \dots, \hat{\alpha}_{n_{m},j}^{(m)}\right)$ is some function of $\hat{\alpha}_{n_{\ell},j}^{(\ell)}, i=1,\dots,m. $
	Obviously, it is reasonable to take $\bar{\alpha}_k=\frac{1}{m}\sum_{1}^{m}\hat{\alpha}_{n_{\ell},j}^{(\ell)}$ or $\bar{\alpha}_k=\hat{\alpha}_{n_{1},j}^{(1)}$ as initial values. Throughout the paper, we write $\stackrel{\mathscr{P}}{\longrightarrow}$ and $\stackrel{\mathscr{F}}{\longrightarrow}$ as the convergence in probability and in distribution, respectively.
\end{rem}

The asymptotically optimal weights given in Proposition \ref{pro1} are not directly usable because several unknown factors $\left(i.e.,\gamma_4^{(\ell)}, \sum_{t=1}^{p}{u^{(\ell)}}_{kt}^4 \right)$ are included, so an estimate of the asymptotically optimal weights is necessary.

\begin{proposition}[Estimation of asymptotically optimal weights]
	\label{pro2}
	Given that Assumption \ref{ass5}--\ref{ass7} hold.  We can obtain the estimation of the asymptotically optimal weights is 
	\begin{equation}
		\hat{\omega}_{\ell}=\frac{n_{\ell}/\hat{\sigma}_{\ell}^2}{\sum_{\ell=1}^{m}n_{\ell}/\hat{\sigma}_i^2}, \quad \ell=1, \dots, m,
		\label{4}
	\end{equation}
	where $$\hat{\sigma}_{\ell}^2:=\left(\hat{\gamma}_4^{(\ell)}-3\right)\bar{\alpha}_k^2\widehat{\sum_{t=1}^{p}{u^{(\ell)}}_{kt}^4}+\frac{2\bar{\alpha}_k^2\left(\bar{\alpha}_k-1\right)^2}{\left(\bar{\alpha}_k-1\right)^2-y_{\ell}},\quad \ell=1, \dots, m,$$
	and where $\hat{\gamma}^{(\ell)}_4$, $\widehat{\sum_{t=1}^{p}{u^{(\ell)}}_{kt}^4}$ can be calculated through the (2.35) in \cite{zhang2022asymptotic} and Lemma \ref{lem3}, respectively. 
\end{proposition}

\begin{rem}
It is important to note that both $\hat{\gamma}^{(\ell)}_4$ and $\widehat{\sum_{t=1}^{p}{u^{(\ell)}}_{kt}^4}$ are estimated using the information from the sample covariance matrix on their respective machines, i.e., $\bm{S}_{n_{\ell}}^{(\ell)}$.
\end{rem}

The following theorem gives the compatibility of asymptotically optimal weight estimates.

\begin{theorem}
	\label{th1}
	Given that Assumption \ref{ass8} hold. As $n_{\ell}\to\infty$, $n=\sum_{{\ell}=1}^{m}n_{\ell}$, there are
	$$\hat{\omega}_{\ell}\stackrel{\mathscr{P}}{\longrightarrow} \omega_{\ell}, \quad {\ell}=1, \dots, m.$$
\end{theorem}

In summary, we have obtained the asymptotically optimal weights and its estimation, and then we can obtain our final weighted estimator in the form of $\sum_{{\ell}=1}^{m}\hat{\omega}_{\ell}\hat{\alpha}_{n_{\ell},j}^{({\ell})}$. For this final weighted estimator, we obtain two important properties of it, compatibility and asymptotic normality, as shown in the following theorem.

\begin{theorem}[Compatibility]
	\label{th2}
	Given that Assumptions \ref{ass4}--\ref{ass8} hold. Under the assumption of the spiked population model in a distributed architecture, as $n_{\ell}\to\infty$, ${\ell}=1, \dots, m$, there are
	$$\sum_{{\ell}=1}^{m}\hat{\omega}_{\ell}\hat{\alpha}_{n_{\ell},j}^{({\ell})}\stackrel{\mathscr{P}}{\longrightarrow}\alpha_k.$$
\end{theorem}

\begin{theorem}[Asymptotic normality]
	\label{th3}
	 Given that Assumptions \ref{ass4}--\ref{ass8} hold. Under the assumption of the spiked population model in a distributed architecture, when $n_{\ell}\to\infty$, ${\ell}=1, \dots, m$, there are
	$$\sqrt{n}\left(\sum_{{\ell}=1}^{m}\hat{\omega}_{\ell}\hat{\alpha}_{n_{\ell},j}^{({\ell})}-\alpha_k\right)\stackrel{\mathscr{F}}{\longrightarrow} \mathcal{N}\left(0,\frac{n}{\sum_{\ell=1}^{m}n_{\ell}/\sigma_{\ell}^2}\right).$$ 
\end{theorem}

In this section, we derived the final form of our weighted estimator along with its two associated statistical properties. It remains to be seen whether the weighted estimator can achieve a statistical error comparable to that of the full sample, potentially with minimal or no loss. This inquiry will be addressed in the following section.

\section{Statistical error analysis\label{sec:4}}

In this section, we analyze the statistical errors of the aforementioned weighted estimators and juxtapose them with the statistical errors of the full sample. Our assessment of statistical errors follows the $\psi_2$-parameter, defined based on sub-Gaussian random variables. Generally speaking, the sub-Gaussian condition is often assumed in the related literature and slightly weaker than the standard normality assumption. For more on sub-Gaussian, please refer to \cite{vershynin2010introduction} for details.

\begin{definition}[Sub-Gaussian Random Variable]
	 The random variable $\bm{X} \in \mathbb{R}$ is called a sub-Gaussian random variable if there exists $C>0$ such that $\left(\mathbb{E}\vert\bm{X}\vert^r\right)^{1/r} \le C\sqrt{r}$ holds. The sub-gaussian norm of $\bm{X}$, denoted $$\Vert\bm{X}\Vert_{\psi_2}=\sup_{r\ge1}r^{-1/2}(\mathbb{E}\vert\bm{X}\vert^r)^{1/r}.$$ 
\end{definition}

\begin{lemma}[Lemma 5.9 in \cite{vershynin2010introduction},Rotation invariance]
	\label{lem1}
	 Consider a finite number of an independent centered sub-gaussian random variable $\bm{X}_i$. Then $\sum_{i} \bm{X}_i $ is also a centered sub-gaussian random variable. Moreover, 
	$$
	\Vert \sum_{i}\bm{X}_i\Vert_{\psi_2}^2\le C \sum_{i}\Vert\bm{X}_i\Vert_{\psi_2}^2
	$$
	where $C$ is an absolute constant.
\end{lemma}

Under the definition of the $\psi_2$-paradigm, we can derive statistical error results for the distributed weighted estimators.

\begin{theorem}\label{th4}
	Given that Assumptions \ref{ass4}--\ref{ass8} hold. As $n_{\ell}\to\infty$, $\ell=1, \dots, m$, there is constants $C$ such that
	$$\Vert\sum_{\ell=1}^{m}\hat{\omega}_{\ell}\hat{\alpha}_{n_{\ell},j}^{(\ell)}-\alpha_k\Vert_{\psi_2}\le \frac{C}{\sqrt{\sum_{\ell=1}^{m}n_{\ell}/\sigma_{\ell}^2}}.$$
\end{theorem}

Naturally, we only take m=1 to get the statistical error obtained by statistical analysis of the full sample data. We denote the estimator for the full sample by $\hat{\alpha}_{n,j}$, then
$$\Vert\hat{\alpha}_{n,j}-\alpha_k\Vert_{\psi_2}\le \frac{C\sigma_{full}}{\sqrt{n}},$$
where $C$ is an absolute constant, $$\sigma_{full}^2=\left(\gamma_4-3\right)\alpha_k^2\sum_{t=1}^{p}u_{kt}^4+\frac{2\alpha_k^2\left(\alpha_k-1\right)^2}{\left(\alpha_k-1\right)^2-y},$$ 
and $n=\sum_{\ell=1}^{m}$, $p/n\to y\in\left(0,1\right)$.

\begin{rem}
	Here we can see when $m=1$, the statistical error rate of order that $\hat{\alpha}_{n,j}$ can be achieved $1/\sqrt{n}$, while at $m>1$, it is clear that the statistical error rate of order for $\sum_{\ell=1}^{m}\hat{\omega}_{\ell}\hat{\alpha}_{n_{\ell},j}^{(\ell)}$ also reaches $1/\sqrt{n}$ because $\sigma_{\ell}^2$ is bounded. This shows that our weighted estimator has the same statistical properties as the full sample estimator. 
\end{rem}

For ease of understanding, we give the following special example. We consider a particular situation from \cite{paul2007asymptotics}.

\begin{example}
	\label{ex1}	
	 Assume that the variables $x_{ij}^{\ell}$ are real Gaussian, and $\bm{\Sigma}$ diagonal whose eigenvalues are all simple. We can determine that the variance of the Gaussian random variable on each machine is 
	$$\sigma_{\ell}^2=\frac{2\alpha_k^2\left(\alpha_k-1\right)^2}{\left(\alpha_k-1\right)^2-y_{\ell}},$$
	and the asymptotically optimal weight is
	\begin{equation}
		\omega_{\ell}=\frac{n_{\ell}(\alpha_k-1)^2-p}{n(\alpha_k-1)^2-mp}.
		\label{8}	
	\end{equation}
	We require a function $h\left(\hat{\alpha}_{n_{\ell},j}^{(1)}, \hat{\alpha}_{n_{\ell},j}^{(2)}, \dots, \hat{\alpha}_{n_{\ell},j}^{(m)}\right)$  to estimate the $\alpha_k$. In this instance, we can set $\bar{\alpha}_k=\frac{1}{m}\sum_{{\ell}=1}^{m}\hat{\alpha}_{n_{\ell},j}^{({\ell})}$. Then, we have	
	$$\sqrt{n}\left(\sum_{{\ell}=1}^{m}\hat{\omega}_{\ell}\hat{\alpha}_{n_{\ell},j}^{({\ell})}-\alpha_k\right)\stackrel{\mathscr{F}}{\longrightarrow} \mathcal{N}\left(0,\frac{2\alpha_k^2(\alpha_k-1)^2}{(\alpha_k-1)^2-my}\right).$$

    And our statistical error result can be expressed as the following equation:
	$$\Vert\sum_{\ell=1}^{m}\hat{\omega}_{\ell}\hat{\alpha}_{n_{\ell},j}^{(\ell)}-\alpha_k\Vert_{\psi_2}\le \frac{C\alpha_k\sqrt{(\alpha_k-1)^2}}{\sqrt{n(\alpha_k-1)^2-mp}}.$$
	for $k\in\left\{1,\dots, M_b\right\}$, $j=k$ and $k\in\left\{M-M_a+1,\dots, M\right\}$, $j=p-M+k$. As for the statistical error for the full sample, we take $m=1$ and get the following equation:  
	
	$$ \Vert\hat{\alpha}_{n,j}-\alpha_k\Vert_{\psi_2}\le \frac{C\alpha_k\sqrt{(\alpha_k-1)^2}}{\sqrt{n(\alpha_k-1)^2-p}}.$$
\end{example}

In this example, it becomes evident that for both cases, where $m=1$ and $m>1$, representing the full-sample estimator and the weighted estimator respectively, they both exhibit a statistical error rate of order $1/\sqrt{n}$. To facilitate practical application, we provide Algorithm \ref{al1}.

\begin{algorithm}[t] 
	\caption{Distributed weighted average estimator} 
	\hspace*{0.02in} {\bf Input:} 
	input function $h\left(\hat{\alpha}_{n_{1},j}^{(1)}, \hat{\alpha}_{n_{2},j}^{(2)}, \dots, \hat{\alpha}_{n_{m},j}^{(m)}\right).$\\
	
	\begin{algorithmic}[1]
		\State On each machine, the covariance matrix of $\bm{S}_{n_{\ell}}^{(\ell)}$ are calculated using \cref{1} 
		\For{i=1:m} 
		\State Calculate the eigenvalues of $\bm{S}_{n_{\ell}}^{(\ell)}$, $\lambda_{n_{\ell},j}$.  
		\State Calculate $\hat{\alpha}_{n_{\ell},j}^{(\ell)}$ by \cref{5} or \cref{7}. 
		\State Send $n_{\ell}$ and $\hat{\alpha}_{n_{\ell},j}^{(\ell)}$ to the central server.
		\EndFor
		\State On the central machine, calculate the initial value
		$\bar{\alpha}_k=h\left(\hat{\alpha}_{n_{1},j}^{(1)}, \hat{\alpha}_{n_{2},j}^{(2)}, \dots, \hat{\alpha}_{n_{m},j}^{(m)}\right)$. Then
		$$\hat{\omega}_{\ell}=\frac{n_{\ell}(\bar{\alpha}_k-1)^2-p}{n(\bar{\alpha}_k-1)^2-mp}.$$	
		\State Calculate the final weighted estimate:
		$$\tilde{\alpha}_k=\sum_{\ell=1}^{m}\hat{\omega}_{\ell}\hat{\alpha}_{n_{\ell},j}^{(\ell)}.$$
	\end{algorithmic}		
	\hspace*{0.02in} {\bf Output:} 
	$\tilde{\alpha}_k.$
	\label{al1}
\end{algorithm}

\begin{rem}
	
As demonstrated by our Algorithm \ref{al1}, we can attain a transfer efficiency of $O(m)$. In contrast, employing the conventional approach of transferring the entire sample covariance matrix would result in a transfer efficiency of $O(mp^2)$. Therefore, our approach signifies a substantial reduction in communication costs.
	
\end{rem}

In the above section, we have analyzed some statistical properties as well as statistical errors of distributed weighted estimators from a theoretical point of view, and next, we will further illustrate the feasibility of weighted estimators from the point of view of simulation experiments.

\section{Simulation studies and empirical analysis\label{sec:5}}

In this section, we conduct a simulation study to demonstrate the efficacy of our distributed spiked eigenvalues estimator. We adopt a specific scenario outlined in \cite{paul2007asymptotics}. We assume the presence of only one spiked eigenvalue in the model, either with $\alpha_1=10$ (representing the largest eigenvalue), or with $\alpha_M=0.01$ (representing the smallest eigenvalue).

\subsection{Statistical error}

Table \ref{tab1} and Table \ref{tab2} present simulations involving spiked eigenvalues of 10 and 0.01, respectively. The reported errors in these tables correspond to the mean square error. The dimensionality is varied for 100, 200, and 300 cases, and the machine number ranges from 50 to 300 at intervals of 6. It is important to highlight that in our methodology, we do not impose a restriction on the number of samples per machine. The sample size on each machine is stochastically generated according to specific rules, resulting in varying sample sizes for each simulation. However, we ensure that both the mean squared error of our statistics and the statistics from the complete sample are computed using the same set of samples. This approach aligns with real-world scenarios where the sample sizes on individual machines tend to differ, lending a practical dimension to our study.

\begin{table}[!ht]
	
	\centering
	
	\caption{The 1,000 mean squared errors in different dimensions and machine numbers, $\alpha_1=10$ (largest eigenvalues). }
	\label{tab1}
	
	\begin{tabular}{c|c|c|c|c|c|c|c|c|c}\hline
		
		\hline
		\diagbox{m}{p}&\multicolumn{3}{c|}{100} & \multicolumn{3}{c|}{200} & \multicolumn{3}{c}{300} \\ 
		\hline
		& \multicolumn{1}{c}{Pooled}& \multicolumn{1}{c}{Weight}& \multicolumn{1}{c|}{Avg}& \multicolumn{1}{c}{Pooled}& \multicolumn{1}{c}{Weight}& \multicolumn{1}{c|}{Avg}& \multicolumn{1}{c}{Pooled}& \multicolumn{1}{c}{Weight}& \multicolumn{1}{c}{Avg}\\ \hline
		\multicolumn{1}{c|}{50}&\multicolumn{1}{c}{6.8317}&\multicolumn{1}{c}{6.847}&\multicolumn{1}{c|}{8.5632}&\multicolumn{1}{c}{5.6439}&\multicolumn{1}{c}{5.6909}&\multicolumn{1}{c|}{6.4932}&\multicolumn{1}{c}{4.2236}&\multicolumn{1}{c}{4.2229}&\multicolumn{1}{c}{4.5984}\\
		\multicolumn{1}{c|}{100}&\multicolumn{1}{c}{3.5407}&\multicolumn{1}{c}{3.5546}&\multicolumn{1}{c|}{4.5878}&\multicolumn{1}{c}{2.9519}&\multicolumn{1}{c}{2.9496}&\multicolumn{1}{c|}{3.2375}&\multicolumn{1}{c}{2.3022}&\multicolumn{1}{c}{2.3161}&\multicolumn{1}{c}{2.4601}\\
		\multicolumn{1}{c|}{150}&\multicolumn{1}{c}{2.4069}&\multicolumn{1}{c}{2.4147}&\multicolumn{1}{c|}{3.2305}&\multicolumn{1}{c}{1.8832}&\multicolumn{1}{c}{1.8943}&\multicolumn{1}{c|}{2.0942}&\multicolumn{1}{c}{1.6679}&\multicolumn{1}{c}{1.6826}&\multicolumn{1}{c}{1.8162}\\
		\multicolumn{1}{c|}{200}&\multicolumn{1}{c}{1.8219}&\multicolumn{1}{c}{1.8267}&\multicolumn{1}{c|}{2.3159}&\multicolumn{1}{c}{1.4424}&\multicolumn{1}{c}{1.4591}&\multicolumn{1}{c|}{1.6973}&\multicolumn{1}{c}{1.1174}&\multicolumn{1}{c}{1.1203}&\multicolumn{1}{c}{1.2131}\\
		\multicolumn{1}{c|}{250}&\multicolumn{1}{c}{1.4299}&\multicolumn{1}{c}{1.4274}&\multicolumn{1}{c|}{1.8654}&\multicolumn{1}{c}{1.1519}&\multicolumn{1}{c}{1.1604}&\multicolumn{1}{c|}{1.2993}&\multicolumn{1}{c}{0.9642}&\multicolumn{1}{c}{0.9708}&\multicolumn{1}{c}{1.052}\\
		\multicolumn{1}{c|}{300}&\multicolumn{1}{c}{1.2867}&\multicolumn{1}{c}{1.2988}&\multicolumn{1}{c|}{1.6278}&\multicolumn{1}{c}{0.9084}&\multicolumn{1}{c}{0.9125}&\multicolumn{1}{c|}{1.0463}&\multicolumn{1}{c}{0.7578}&\multicolumn{1}{c}{0.7661}&\multicolumn{1}{c}{0.8289}\\
		\hline \hline 
	\end{tabular}	
	\begin{tablenotes}
		\footnotesize
		\centering
		\item Note: The error values in Table \ref{tab1} are multiplied by $10^{-3}$. 
	\end{tablenotes}

\end{table}

\begin{table}[!ht]
	
	\centering
	
	\caption{The 1,000 mean squared errors in different dimensions and machine numbers, $\alpha_M=0.01$ (smallest eigenvalues). }
	\label{tab2}

	\begin{tabular}{c|c|c|c|c|c|c|c|c|c}\hline
		
		\hline
		\diagbox{m}{p}&\multicolumn{3}{c|}{100} & \multicolumn{3}{c|}{200} & \multicolumn{3}{c}{300} \\ 
		\hline
		& \multicolumn{1}{c}{Pooled}& \multicolumn{1}{c}{Weight}& \multicolumn{1}{c|}{Avg}& \multicolumn{1}{c}{Pooled}& \multicolumn{1}{c}{Weight}& \multicolumn{1}{c|}{Avg}& \multicolumn{1}{c}{Pooled}& \multicolumn{1}{c}{Weight}& \multicolumn{1}{c}{Avg}\\ \hline
		\multicolumn{1}{c|}{50}&\multicolumn{1}{c}{7.3351}&\multicolumn{1}{c}{9.3872}&\multicolumn{1}{c|}{15.053}&\multicolumn{1}{c}{5.9349}&\multicolumn{1}{c}{8.5021}&\multicolumn{1}{c|}{12.269}&\multicolumn{1}{c}{4.861}&\multicolumn{1}{c}{7.814}&\multicolumn{1}{c}{9.9585}\\
		\multicolumn{1}{c|}{100}&\multicolumn{1}{c}{3.9825}&\multicolumn{1}{c}{5.1609}&\multicolumn{1}{c|}{8.5341}&\multicolumn{1}{c}{2.9143}&\multicolumn{1}{c}{4.6961}&\multicolumn{1}{c|}{6.6514}&\multicolumn{1}{c}{2.4175}&\multicolumn{1}{c}{3.9242}&\multicolumn{1}{c}{4.921}\\
		\multicolumn{1}{c|}{150}&\multicolumn{1}{c}{2.4809}&\multicolumn{1}{c}{3.3616}&\multicolumn{1}{c|}{6.0184}&\multicolumn{1}{c}{1.9868}&\multicolumn{1}{c}{3.277}&\multicolumn{1}{c|}{4.7696}&\multicolumn{1}{c}{1.5963}&\multicolumn{1}{c}{2.7296}&\multicolumn{1}{c}{3.5621}\\
		\multicolumn{1}{c|}{200}&\multicolumn{1}{c}{1.7882}&\multicolumn{1}{c}{2.7705}&\multicolumn{1}{c|}{5.0585}&\multicolumn{1}{c}{1.428}&\multicolumn{1}{c}{2.3824}&\multicolumn{1}{c|}{3.4623}&\multicolumn{1}{c}{1.1311}&\multicolumn{1}{c}{2.1085}&\multicolumn{1}{c}{2.9557}\\
		\multicolumn{1}{c|}{250}&\multicolumn{1}{c}{1.3775}&\multicolumn{1}{c}{2.2973}&\multicolumn{1}{c|}{4.3883}&\multicolumn{1}{c}{1.1605}&\multicolumn{1}{c}{2.0506}&\multicolumn{1}{c|}{2.9741}&\multicolumn{1}{c}{0.8849}&\multicolumn{1}{c}{1.7931}&\multicolumn{1}{c}{2.3508}\\
		\multicolumn{1}{c|}{300}&\multicolumn{1}{c}{1.1395}&\multicolumn{1}{c}{1.96}&\multicolumn{1}{c|}{3.7089}&\multicolumn{1}{c}{0.9397}&\multicolumn{1}{c}{1.7149}&\multicolumn{1}{c|}{2.5958}&\multicolumn{1}{c}{0.7821}&\multicolumn{1}{c}{1.5458}&\multicolumn{1}{c}{2.1578}\\
		\hline \hline 
	\end{tabular}
	\begin{tablenotes}
		\footnotesize
		\centering
		\item Note: The error values in Table \ref{tab2} are multiplied by $10^{-9}$. 
	\end{tablenotes}

\end{table}

\begin{rem}
	Here, we use "Pooled" for the full sample estimate, "Weight" for our weighted estimate, and "Avg" for the mean estimate (i.e., the weights are taken directly as 1/m).
\end{rem}

From Table \ref{tab1} and Table \ref{tab2}, we observe that our weighted estimates for $\alpha_k$ are slightly less accurate than the full sample estimates, but this discrepancy diminishes as both the number of dimensions and machines increase. It is worth noting that the errors are on the order of $10^{-3}$ for the larger spiked eigenvalues and around $10^{-9}$ for the smaller ones, consistent with the findings in our Theorem \ref{th4}.

Next, we will examine the impact of dimensionality and the number of machines on the estimation error.

\subsection{Effect of machine number}

\begin{figure}[htbp]
	\centering
	\subfigure[p=100]{
		\begin{minipage}[t]{0.3\linewidth}
			\includegraphics[scale=0.4]{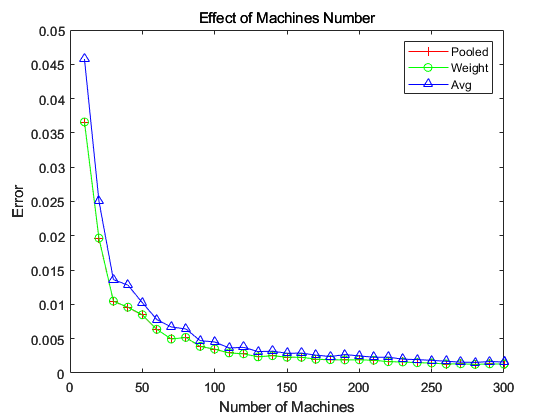}
		\end{minipage}
		\label{penG1}
	}
	\subfigure[p=200]{
		\begin{minipage}[t]{0.3\linewidth}
			\includegraphics[scale=0.4]{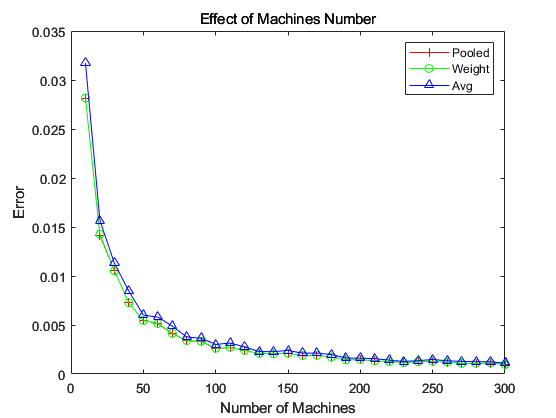}
		\end{minipage}
		\label{penG2}
	}
	\subfigure[p=300]{
		\begin{minipage}[t]{0.3\linewidth}
			\includegraphics[scale=0.4]{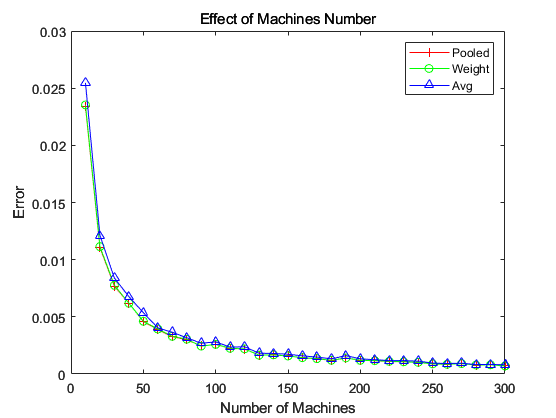}
		\end{minipage}
		\label{penG3}
	}
	\caption{The effect of machine number on the estimation of  largest spiked eigenvalues.}
	\label{key1}
\end{figure}

\begin{figure}[htbp]
	\centering
	\subfigure[p=100]{
		\begin{minipage}[t]{0.3\linewidth}
			\includegraphics[scale=0.4]{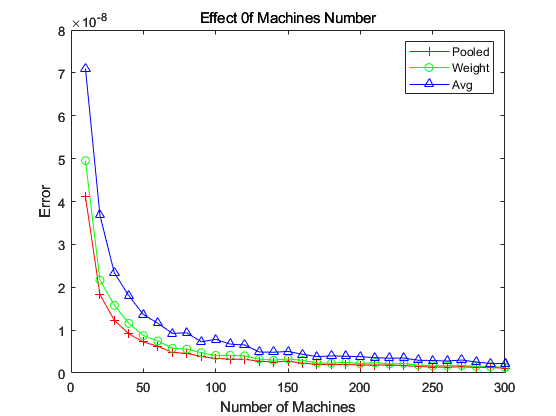}
		\end{minipage}
		\label{penG4}
	}
	\subfigure[p=200]{
		\begin{minipage}[t]{0.3\linewidth}
			\includegraphics[scale=0.4]{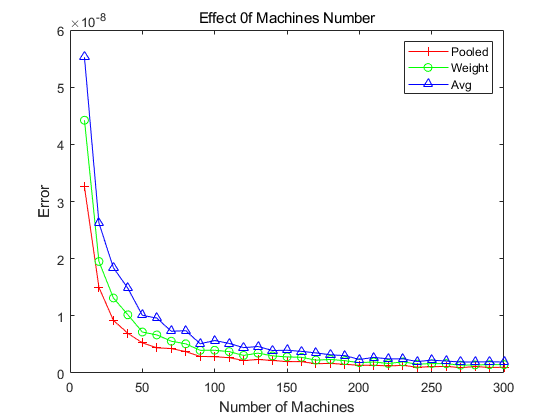}
		\end{minipage}
		\label{penG5}
	}
	\subfigure[p=300]{
		\begin{minipage}[t]{0.3\linewidth}
			\includegraphics[scale=0.4]{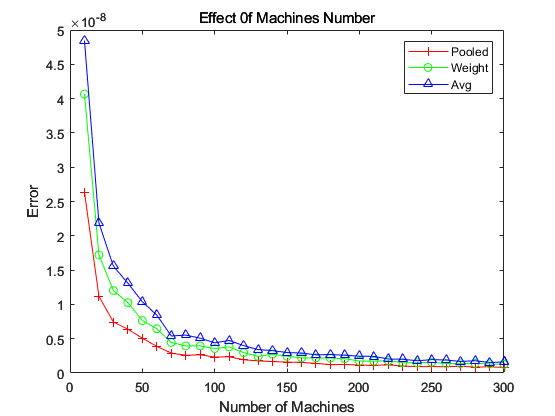}
		\end{minipage}
		\label{penG6}
	}
	\caption{The effect of machine number on the estimation of  smallest spiked eigenvalues.}
	\label{key2}
\end{figure}

In this subsection, we provide a more intuitive analysis of how the number of machines affects the accuracy of the three estimators, with fixed dimensions of 100, 200, and 300. Figure \ref{key1} illustrates the variation of the largest spiked eigenvalues with the number of machines under constant dimensionality. Notably, the curves of the Weighted-estimate and the Pooled-estimate closely align and consistently fall below the curve of the Average-estimate. Figure \ref{key2} depicts the variation of the smallest spiked eigenvalues with machine numbers in a fixed dimension, which complements the data in our table. Both plots demonstrate that the error in all three estimators diminishes as the number of machines increases. Although the outcomes are similar, their underlying principles differ. As the number of machines grows, the total sample size expands, which is the driving force behind the reduction in Pooled-estimate error. We ensure that the number of samples on each machine surpasses the number of dimensions, a prerequisite for our theorem to hold. Consequently, the error in both the weighted and mean estimates diminishes with increasing machine count. In contrast to traditional distributed algorithms, we advocate having as many machines as necessary to meet our conditions, as this approach yields superior estimation.

\subsection{Effect of Dimensionality }

\begin{figure}
	\centering
	\subfigure[m=100]{
		\begin{minipage}[t]{0.3\linewidth}
			\includegraphics[scale=0.4]{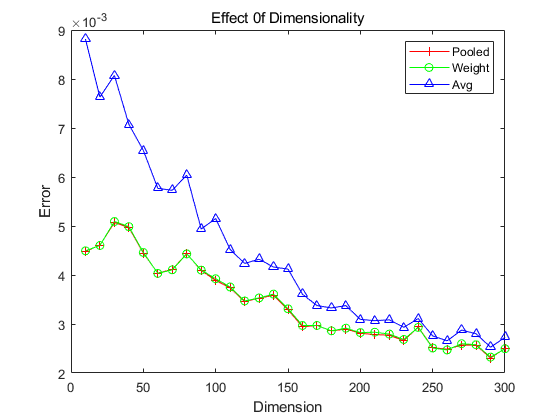}
		\end{minipage}
		\label{m1}
	}
	\subfigure[m=200]{
		\begin{minipage}[t]{0.3\linewidth}
			\includegraphics[scale=0.4]{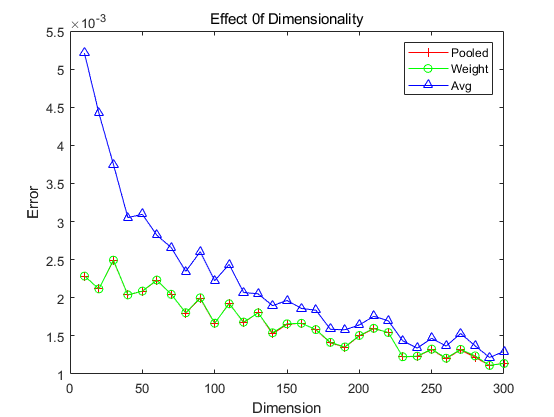}
		\end{minipage}
		\label{m2}
	}
	\subfigure[m=300]{
		\begin{minipage}[t]{0.3\linewidth}
			\includegraphics[scale=0.4]{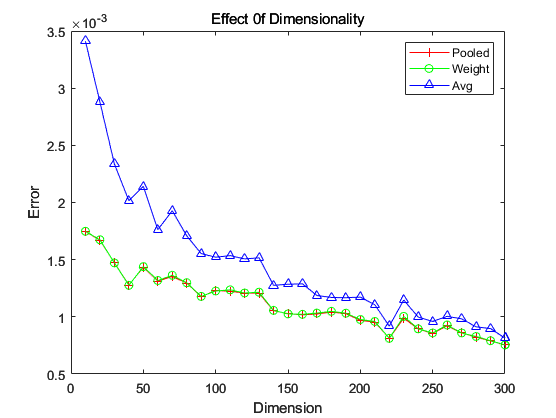}
		\end{minipage}
		\label{m3}
	}
	\caption{The effect of dimensionality on the estimation of largest spiked eigenvalues.}
	\label{key3}
\end{figure}

\begin{figure}
	\centering
	\subfigure[m=100]{
		\begin{minipage}[t]{0.3\linewidth}
			\includegraphics[scale=0.4]{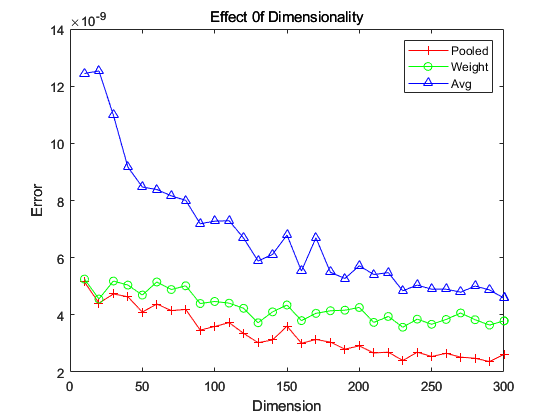}
		\end{minipage}
		\label{sm1}
	}
	\subfigure[m=200]{
		\begin{minipage}[t]{0.3\linewidth}
			\includegraphics[scale=0.4]{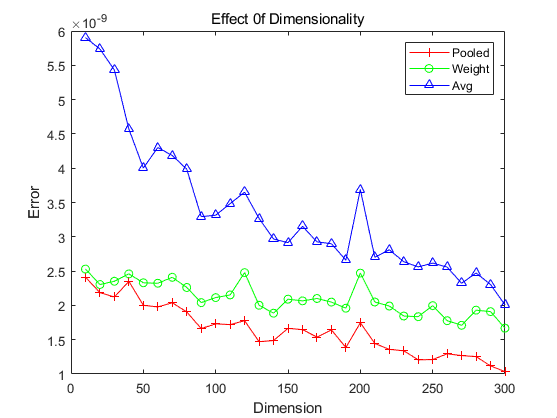}
		\end{minipage}
		\label{sm2}
	}
	\subfigure[m=300]{
		\begin{minipage}[t]{0.3\linewidth}
			\includegraphics[scale=0.4]{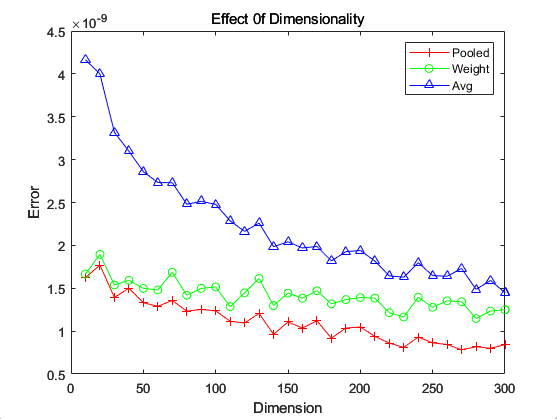}
		\end{minipage}
		\label{sm3}
	}
	\caption{The effect of dimensionality on the estimation of smallest spiked eigenvalues.}
	\label{key4}
\end{figure}

In this subsection, we will provide a more intuitive analysis of the impact of dimensionality on the error of the three estimators with a fixed number of machines (100, 200, and 300). While Figures \ref{key1} and \ref{key2} may not distinctly highlight the difference between the weighted and Avg-estimate, Figures \ref{key3} and \ref{key4} clearly demonstrate the sensitivity of these estimators to dimensionality. Notably, the weighted estimate exhibits greater stability and proximity to the pooled estimate compared to the average estimate. In summary, our weighted estimates, though with a marginal loss of precision, achieve higher transmission efficiency than the pooled estimates, while also being more accurate and stable than the average estimates.

\subsection{Selection of initial values}

In this subsection, we investigate the impact of initial value selection on the final error results. We consider three different initial values for the simulations. The first initial value fixes the estimate on the first server (effectively the central server), denoted as $\bar{\alpha}_k^1=\hat{\alpha}_{n_{1},j}^{(1)}$. The second initial value is chosen to be the estimate that deviates the furthest from the true value $\alpha_k$, i.e., $\bar{\alpha}_k^2=\max\left\{\left(\hat{\alpha}_{n_{\ell},j}^{(\ell)}-\alpha_k\right)^2, i=1,\dots,m.\right\}$. The third initial value is the average estimate, denoted as $\bar{\alpha}_k^3=\frac{1}{m}\sum{1}^{m}\hat{\alpha}_{n_{\ell},j}^{(\ell)}$.

\begin{table}[!ht]
	
	\centering
	
	\caption{The 1,000 mean squared errors in different dimensions and machine numbers, $\alpha_1=10$ (largest eigenvalues). }
	\label{tab3}
	
	\begin{tabular}{c|c|c|c|c|c|c|c|c|c}\hline
		
		\hline
		\diagbox{m}{p}&\multicolumn{3}{c|}{100} & \multicolumn{3}{c|}{200} & \multicolumn{3}{c}{300} \\ 
		\hline
		& \multicolumn{1}{c}{$\bar{\alpha}_k^1$}& \multicolumn{1}{c}{$\bar{\alpha}_k^2$}& \multicolumn{1}{c|}{$\bar{\alpha}_k^3$}& \multicolumn{1}{c}{$\bar{\alpha}_k^1$}& \multicolumn{1}{c}{$\bar{\alpha}_k^2$}& \multicolumn{1}{c|}{$\bar{\alpha}_k^3$}& \multicolumn{1}{c}{$\bar{\alpha}_k^1$}& \multicolumn{1}{c}{$\bar{\alpha}_k^2$}& \multicolumn{1}{c}{$\bar{\alpha}_k^3$}\\ \hline
		\multicolumn{1}{c|}{50}&\multicolumn{1}{c}{7.2768}&\multicolumn{1}{c}{7.2766}&\multicolumn{1}{c|}{7.2767}&\multicolumn{1}{c}{5.9128}&\multicolumn{1}{c}{5.9127}&\multicolumn{1}{c|}{5.9129}&\multicolumn{1}{c}{4.4352}&\multicolumn{1}{c}{4.435}&\multicolumn{1}{c}{4.4352}\\
		\multicolumn{1}{c|}{100}&\multicolumn{1}{c}{3.6246}&\multicolumn{1}{c}{3.6245}&\multicolumn{1}{c|}{3.6246}&\multicolumn{1}{c}{2.9486}&\multicolumn{1}{c}{2.9485}&\multicolumn{1}{c|}{2.9485}&\multicolumn{1}{c}{2.3561}&\multicolumn{1}{c}{2.3561}&\multicolumn{1}{c}{2.3561}\\
		\multicolumn{1}{c|}{150}&\multicolumn{1}{c}{2.3584}&\multicolumn{1}{c}{2.3583}&\multicolumn{1}{c|}{2.3584}&\multicolumn{1}{c}{1.9445}&\multicolumn{1}{c}{1.9445}&\multicolumn{1}{c|}{1.9445}&\multicolumn{1}{c}{1.6778}&\multicolumn{1}{c}{1.6779}&\multicolumn{1}{c}{1.6779}\\
		\multicolumn{1}{c|}{200}&\multicolumn{1}{c}{1.8892}&\multicolumn{1}{c}{1.8892}&\multicolumn{1}{c|}{1.8892}&\multicolumn{1}{c}{1.505}&\multicolumn{1}{c}{1.5049}&\multicolumn{1}{c|}{1.505}&\multicolumn{1}{c}{1.1821}&\multicolumn{1}{c}{1.1821}&\multicolumn{1}{c}{1.1821}\\
		\multicolumn{1}{c|}{250}&\multicolumn{1}{c}{1.4567}&\multicolumn{1}{c}{1.4567}&\multicolumn{1}{c|}{1.4568}&\multicolumn{1}{c}{1.1046}&\multicolumn{1}{c}{1.1047}&\multicolumn{1}{c|}{1.1046}&\multicolumn{1}{c}{0.8617}&\multicolumn{1}{c}{0.8617}&\multicolumn{1}{c}{0.8616}\\
		\multicolumn{1}{c|}{300}&\multicolumn{1}{c}{1.1598}&\multicolumn{1}{c}{1.1597}&\multicolumn{1}{c|}{1.1598}&\multicolumn{1}{c}{0.857}&\multicolumn{1}{c}{0.8569}&\multicolumn{1}{c|}{0.857}&\multicolumn{1}{c}{0.7858}&\multicolumn{1}{c}{0.7858}&\multicolumn{1}{c}{0.7858}\\
		\hline \hline 
	\end{tabular}	
	\begin{tablenotes}
		\footnotesize
		\centering
		\item Note: The error values in Table \ref{tab3} are multiplied by $10^{-3}$. 
	\end{tablenotes}

\end{table}

In the selection of initial values, we observe that as long as Assumption \ref{ass8} holds, the final error results tend to converge. This is evident in Table \ref{tab3}, where even when choosing the estimate farthest from the true value as the initial value, the overall statistical error remains consistent. In practical scenarios with a large number of machines, one can directly adopt the upper estimate from the central server as the initial value for weight calculation. For a more cautious approach, the mean estimate can be utilized as the initial value, incurring no additional communication cost, as it only involves a summation and averaging step on the central server. Algorithm \ref{al1} presented here employs the mean estimate as the initial value.

\subsection{Asymptotic normality.}

\begin{figure}
	\centering
	\subfigure[m=100, p=100]{
		\begin{minipage}[t]{0.45\linewidth}
			\includegraphics[scale=0.55]{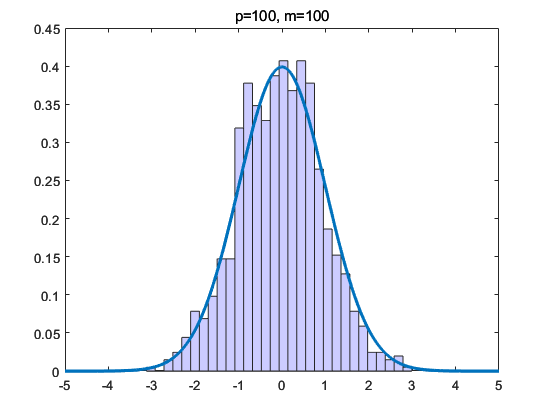}
		\end{minipage}
	}
	\subfigure[m=200, p=200]{
		\begin{minipage}[t]{0.45\linewidth}
			\includegraphics[scale=0.55]{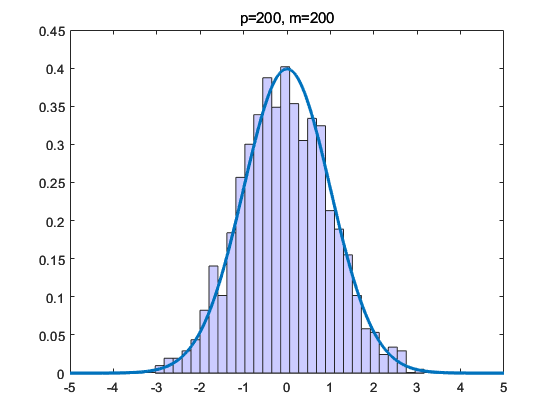}
		\end{minipage}
	}
	\caption{Plot of density function of largest spiked eigenvalues.}
	\label{key5}
\end{figure}

\begin{figure}[htbp]
	\centering
	\subfigure[m=100, p=100]{
		\begin{minipage}[t]{0.45\linewidth}
			\includegraphics[scale=0.55]{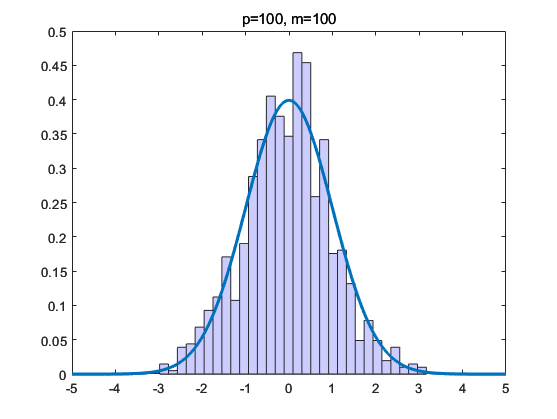}
		\end{minipage}
	}
	\subfigure[m=200, p=200]{
		\begin{minipage}[t]{0.45\linewidth}
			\includegraphics[scale=0.55]{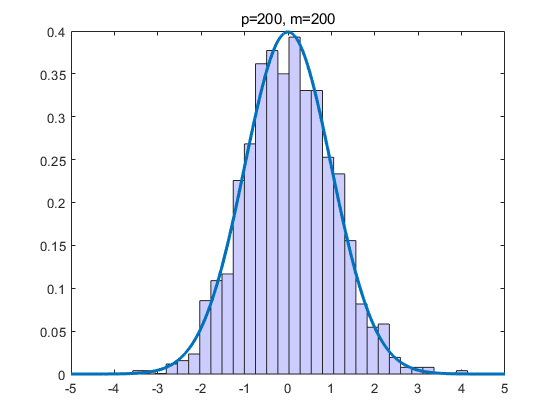}
		\end{minipage}
	}
	\caption{Plot of density function of smallest spiked eigenvalues.}
	\label{key6}
\end{figure}

This subsection presents a simulation study focusing on the Asymptotic normality. In Figure \ref{key5} and \ref{key6}, the blue curve depicts the standard normal distribution density. Two histograms were generated for the largest and smallest spiked eigenvalues, considering $m=100$, $p=100$, and $m=200$, $p=200$, respectively. The figures clearly suggest that our weighted estimates closely align with a normal distribution function. For brevity and to avoid redundancy, we omit the histograms for other parameter combinations, which exhibit similar behavior.

\subsection{Real data analysis}

In this subsection, we employ practical examples to elucidate the issue. The studies by \cite{cinar2019classification}, \cite{koklu2021classification}, \cite{cinar2021determination}, and \cite{cinar2022identification} provided us with the data sources, which can be accessed at \href{https://www.muratkoklu.com/datasets/}{https://www.muratkoklu.com/datasets/}. We acquired 75,000 data instances from their research, encompassing 106 features of rice. Our objective was to ascertain their maximum eigenvalues for classification purposes. We applied three methods to compute these maximum eigenvalues: the pooled method, the weighted method, and the average method. As depicted in Figure \ref{real}, the weighted method exhibits greater stability and proximity to the pooled estimate compared to the average method. Given our relatively modest dimensionality, the sample size on each machine tends to be small as the number of machines increases, while the total sample size and dimensionality remain constant. Notably, at a machine count of 20, our weighted estimate closely approximates the full sample estimate, while the mean estimate displays more variability. We posit that provided the sample size on each machine is sufficiently large, our weighted estimation can effectively demonstrate its superiority. Consequently, when working with large-scale data, utilizing our weighted estimator is the more favorable option.

\begin{figure}[htbp]
	\centering
	\includegraphics[scale=0.38]{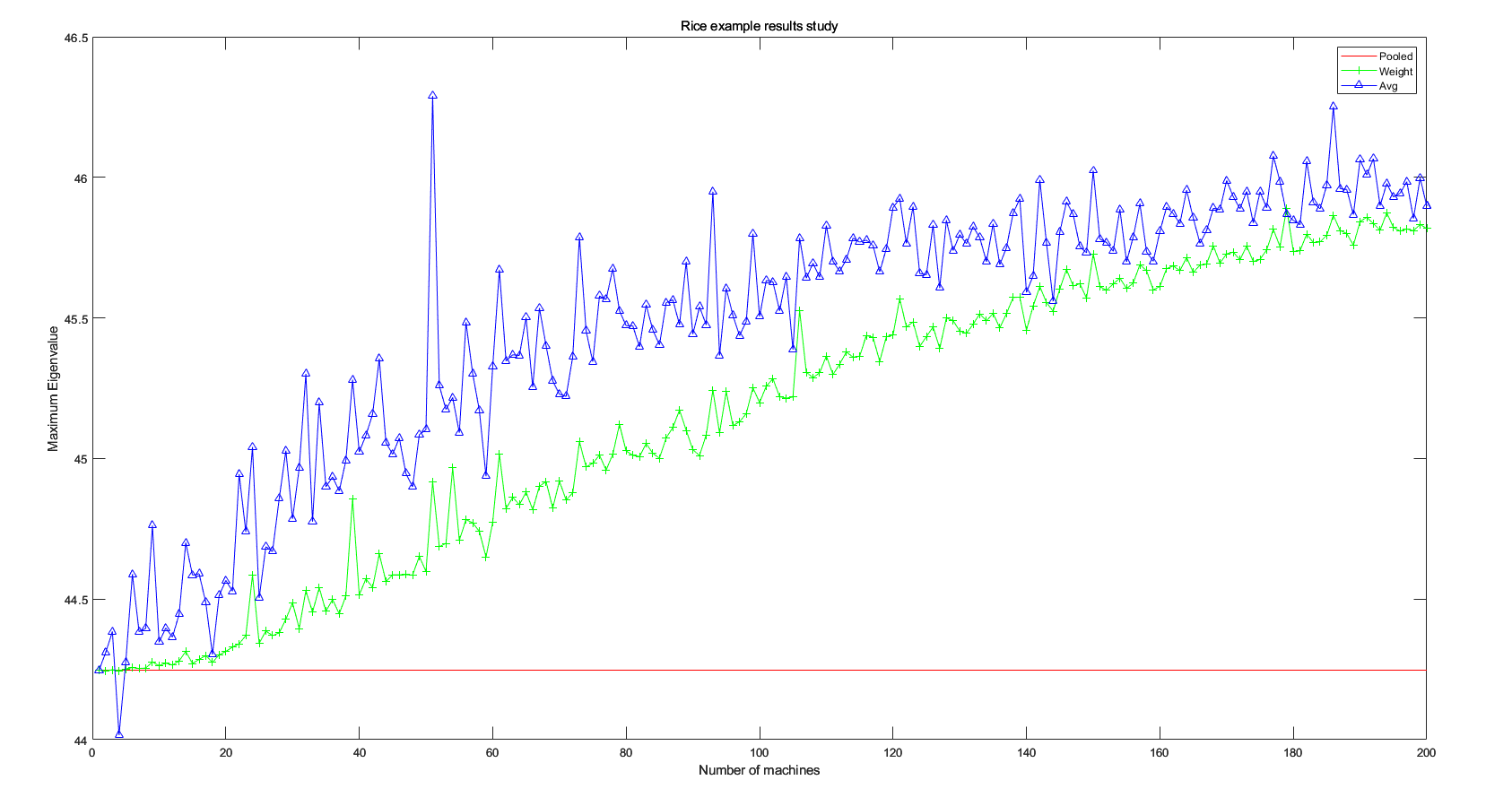}
	\caption{List of maximum eigenvalue calculations for 75,000 rice samples. }
	\label{real}
\end{figure}
\section{Discussion\label{sec:6}}

Our theoretical results rely on the assumption that the fourth-order moments of the samples are finite, thus accommodating data with heavier tails compared to the sub-Gaussian sample assumption. In addition, we specify that the samples are only i.i.d. on each machine, thus accounting for heterogeneity between machines. Another small advantage of our weighted estimator is that iteration is not required, thus greatly reducing the communication cost. In addition, the weighted estimator exhibits superior performance on large-scale data, with lower transmission overhead (compared to full-sample estimation), higher and more stable estimation accuracy (compared to average estimation), and significantly fewer restrictions on the number of machines.

Furthermore, the weighted estimator proposed in this paper exhibits certain limitations. As the number of machines approaches infinity and consequently, $p/n$ tends to 0, the full-sample estimator has the capability to estimate all non-zero and non-one eigenvalues. In contrast, the weighted estimator is restricted to estimating only those eigenvalues falling outside the interval $\left[1-\sqrt{y_{\ell}}, 1+\sqrt{y_{\ell}}\right]$, where $\ell=1, \dots,m$. This gives rise to a limitation, specifically that eigenvalues within the interval $\left[1-\mathop{\max}\limits_{\ell\in{1,\dots,m}}\left\{\sqrt{y_{\ell}}\right\}, 1+\mathop{\max}\limits_{\ell\in{1,\dots,m}}\left\{\sqrt{y_{\ell}}\right\}\right]$ of the removal point 1, cannot be estimated by the weighted estimator. However, it is important to note that in practical applications, the weighted estimator proves effective in resolving the majority of eigenvalue estimation challenges. Additionally, while our paper assumes knowledge of the number of spiked eigenvalues, this may not always be feasible in real-world scenarios. Therefore, the estimation of the number of spiked eigenvalues is an important consideration.

This marks our initial endeavor to bridge the theory of large-dimensional random matrices with machine learning algorithms. Despite one being theoretical and the other practical, we believe there is a need for them to converge. Although they represent two distinct modes of thinking, they converge towards a common objective: the more efficient analysis of high-dimensional, large-scale datasets. Both domains already house outstanding ideas and conclusions, and their strategic integration may give rise to even more brilliant insights.

\section*{Acknowledgments}

 Jiang Hu was partially supported by NSFC Grants No. 12171078, No. 12292980, No. 12292982,  and Fundamental Research Funds for the Central Universities No. 2412023YQ003.

\section*{Appendix}
In the appendix, we states the proofs of  Proposition \ref{pro1} and Theorems \ref{th1}--\ref{th4}, respectively.

\begin{proof}[Proof of Proposition \ref{pro1}]
	We need to get:
	$$\omega_{\ell}=\arg\min_{\omega_{\ell}^*}\mathbb{E}\left(\sum_{\ell=1}^{m}\omega_{\ell}^*\hat{\alpha}_{n_{\ell},j}^{({\ell})}-\alpha_k\right)^2, \quad {\ell}=1,\dots,m.$$	
	Using the method of Lagrange multipliers, under constraint $\sum_{{\ell}=1}^{m}\omega_{\ell}=1$, there are
	
	\begin{align*}
		L_n\left(\omega_1,\dots,\omega_m;\lambda\right)=&\mathbb{E}\left(\sum_{{\ell}=1}^{m}\omega_{\ell}\hat{\alpha}_{n_{\ell},j}^{({\ell})}-\alpha_k\right)^2+2\lambda\left(\sum_{{\ell}=1}^{m}\omega_{\ell}-1\right)\\
		=&\sum_{{\ell}=1}^{m}\omega_{\ell}^2\mathbb{E}\left(\hat{\alpha}_{n_{\ell},j}^{({\ell})}-\alpha_k\right)^2+2\lambda\left(\sum_{{\ell}=1}^{m}\omega_{\ell}-1\right).
	\end{align*}
	The function $L_n\left(\omega_1,\dots,\omega_m;\lambda\right)$ takes the partial derivatives for $\omega_{\ell}$, ${\ell}=1,\dots,m$, and $\lambda$, respectively:
	\begin{equation*}
		\begin{cases}
			 \frac{\partial L_n\left(\omega_1,\dots,\omega_m;\lambda\right)}{\partial \omega_1}&=2\omega_1\mathbb{E}\left(\hat{\alpha}_{n_1,j}^{(1)}-\alpha_k\right)^2+2\lambda=0 ,\\
			 &\vdots \\
			 \frac{\partial L_n\left(\omega_1,\dots,\omega_m;\lambda\right)}{\partial \omega_m}&=2\omega_m\mathbb{E}\left(\hat{\alpha}_{n_m,j}^{(m)}-\alpha_k\right)^2+2\lambda=0 ,\\
			\frac{\partial L_n\left(\omega_1,\dots,\omega_m;\lambda\right)}{\partial \lambda}&=\sum_{{\ell}=1}^{m}\omega_{\ell}-1=0. 
		\end{cases}
	\end{equation*}
	Writing $\mathbb{E}\left(\hat{\alpha}_{n_{\ell},j}^{({\ell})}-\alpha_k\right)^2$ as $\tilde{\sigma}_{\ell}^2$, 	according to Theorem \ref{th5}, \emph{Helly-Bray Theorem},  and the fact that   	$$\omega_{\ell}^*=\frac{1/\tilde{\sigma}_{\ell}^2}{\sum_{i=1}^{m}1/\tilde{\sigma}_{i}^2},\quad \ell=1,\dots,m, \quad \lambda=\frac{1}{\sum_{\ell=1}^{m}1/\tilde{\sigma}_{\ell}^2}, $$
we have	
	$$\tilde{\sigma}_{\ell}^2 \to \sigma_{\ell}^2/n_{\ell}=\left(\gamma_4^{(\ell)}-3\right)\frac{\alpha_k^2}{n_{\ell}}\sum_{t=1}^{p}{u^{(\ell)}}_{kt}^4+\frac{2\alpha_k^2\left(\alpha_k-1\right)^2}{n_\ell\left(\alpha_k-1\right)^2-p},\quad \text{as}\quad n_{\ell}\to\infty.$$
	Then by the \emph{Continuous Mapping Theorem}, we can get the asymptotically optimal weights as
	$$\omega_{\ell}=\frac{n_{\ell}/\sigma_{\ell}^2}{\sum_{i=1}^{m}n_{i}/\sigma_i^2}, \quad \ell=1, \dots, m.$$
	And the mean square error results as flowing,
	\begin{align*}
		&\mathbb{E}\left(\sum_{{\ell}=1}^{m}\omega_{\ell}\hat{\alpha}_{n_{\ell},j}^{({\ell})}-\alpha_k\right)^2=\sum_{\ell=1}^{m}\omega_{\ell}^2\mathbb{E}\left(\hat{\alpha}_{n_{\ell},j}^{(\ell)}-\alpha_k\right)^2\\
		=&\sum_{\ell=1}^{m}\left(\frac{1/\tilde{\sigma}_{\ell}^2}{\sum_{\ell=1}^{m}1/\tilde{\sigma}_{\ell}^2}\right)^2\tilde{\sigma}_{\ell}^2
		=\frac{1}{\sum_{\ell=1}^{m}1/\tilde{\sigma}_{\ell}^2}
		\to\frac{1}{\sum_{\ell=1}^{m}n_{\ell}/\sigma_{\ell}^2},\quad \text{as} \quad n_{\ell}\to\infty.
	\end{align*}	
	Then we complete the proof of Proposition \ref{pro1}.
\end{proof}

\begin{proof}[Proof of Theorem \ref{th1}]
	It is known that $\sigma_{\ell}^2$ is a continuous function with respect to $\alpha_k$, $\sum_{t=1}^{p}{u^{(\ell)}}_{kt}^4$ and $\gamma_4^{(\ell)}$, respectively. According to Lemma \ref{lem3}  and Eq. (2.35) in \cite{zhang2022asymptotic}, we know that $\widehat{\sum_{t=1}^{p}{u^{(\ell)}}_{kt}^4}$ and $\hat{\gamma}_4^{(\ell)}$ are consistently estimates of $\sum_{t=1}^{p}{u^{(\ell)}}_{kt}^4$ and $\gamma_4^{(\ell)}$, respectively. If Assumption \ref{ass6} holds, by virtue of the Continuous Mapping Theorem, the theorem is established. 
\end{proof}

\begin{proof}[Proof of Theorem \ref{th2}]
	Note that
	$$	\sum_{\ell=1}^{m}\hat{\omega}_{\ell}\hat{\alpha}_{n_{\ell},j}^{({\ell})}-\alpha_k=\sum_{{\ell}=1}^{m}\left(\hat{\omega}_{\ell}-\omega_{\ell}\right)\left(\hat{\alpha}_{n_{\ell},j}^{({\ell})}-\alpha_k\right)+\sum_{{\ell}=1}^{m}\omega_{\ell}\left(\hat{\alpha}_{n_{\ell},j}^{({\ell})}-\alpha_k\right).$$
	We have established that $\hat{\alpha}_{n_{\ell},j}^{({\ell})}\to\alpha_k$, almost surely. Furthermore, according to Theorem \ref{th1}, $\hat{\omega}_{\ell}\stackrel{\mathscr{P}}{\longrightarrow}{\omega}_{\ell}$, where $0<\omega_{\ell}<1$ is a non-random variable. Applying the \emph{Slutsky Theorems}, we deduce $$\sum_{\ell=1}^{m}\left(\hat{\omega}_{\ell}-\omega_{\ell}\right)\left(\hat{\alpha}_{n_{\ell},j}^{({\ell})}-\alpha_k\right)\stackrel{\mathscr{P}}{\longrightarrow}0,$$ 
	$$\sum_{{\ell}=1}^{m}\omega_{\ell}\left(\hat{\alpha}_{n_{\ell},j}^{({\ell})}-\alpha_k\right)\to 0, \quad \text{almost} \ \text{surely},$$
	thus concluding
    $\sum_{\ell=1}^{m}\hat{\omega}_{\ell}\hat{\alpha}_{n_{\ell},j}^{(\ell)}\stackrel{\mathscr{P}}{\longrightarrow}\alpha_k.$
\end{proof}

\begin{proof}[Proof of Theorem \ref{th3}]It follows that
	\begin{align*}
		&\sqrt{n}\left(\sum_{\ell=1}^{m}\hat{\omega}_{\ell} \hat{\alpha}_{n_{\ell},j}^{(\ell)}-\alpha_k\right)\\
		=&\sqrt{n}\left(\sum_{\ell=1}^{m}\left(\hat{\omega}_{\ell} -\omega_{\ell}\right)\left(\hat{\alpha}_{n_{\ell},j}^{(\ell)}-\alpha_k\right)\right)+\sqrt{n}\sum_{\ell=1}^{m}\omega_{\ell} \left(\hat{\alpha}_{n_{\ell},j}^{(\ell)}-\alpha_k\right)\\
		=&\sum_{\ell=1}^{m}\left(\hat{\omega}_{\ell} -\omega_{\ell}\right)\sqrt{\frac{n}{n_{\ell}}} \sqrt{n_{\ell}}\left(\hat{\alpha}_{n_{\ell},j}^{(\ell)}-\alpha_k\right)+\sum_{\ell=1}^{m}\omega_{\ell} \sqrt{\frac{n}{n_{\ell}}} \sqrt{n_{\ell}}\left(\hat{\alpha}_{n_{\ell},j}^{(\ell)}-\alpha_k\right)\\
		\doteq & \bm{I}+\bm{II}.
			\end{align*}	
	Let $Z_{n_{\ell}}=\sqrt{n_{\ell}}\left(\hat{\alpha}_{n_{\ell},j}^{(\ell)}-\alpha_k\right)$, $\ell=1, \dots, m$. By Theorem \ref{th5}, on each machine we have
	$$Z_{n_{\ell}}=\sqrt{n_{\ell}}\left(\hat{\alpha}_{n_{\ell},j}^{(\ell)}-\alpha_k\right)\stackrel{\mathscr{F}}{\longrightarrow}\mathcal{N}\left(0, \sigma_{\ell}^2\right),$$ 
Because $\hat{\omega}_{\ell}\stackrel{\mathscr{P}}{\longrightarrow}\omega_{\ell}$, $i=1, \dots, m$,
	by the Slutsky theorem, we have that 
	$$\bm{I}=\sum_{\ell=1}^{m}\left(\hat{\omega}_{\ell} -\omega_{\ell}\right)\sqrt{\frac{n}{n_{\ell}}} Z_{n_{\ell}}\stackrel{\mathscr{F}}{\longrightarrow}0$$
	and
	$$\bm{II}=\sum_{\ell=1}^{m}\omega_{\ell} \sqrt{\frac{n}{n_{\ell}}} Z_{n_{\ell}}\stackrel{\mathscr{F}}{\longrightarrow}\mathcal{N}\left(0, \sum_{\ell=1}^{m}\omega_{\ell}^2\frac{n}{n_{\ell}}\sigma_{\ell}^2\right).$$
	Thus, we can conclude that	$$\sqrt{n}\left(\sum_{\ell=1}^{m}\hat{\omega}_{\ell}\hat{\alpha}_{n_{\ell},j}^{(\ell)}-\alpha_k\right)=\bm{I}+\bm{II}\stackrel{\mathscr{F}}{\longrightarrow} \mathcal{N}\left(0, \frac{n }{\sum_{\ell=1}^{m}n_{\ell}/\sigma_{\ell}^2}\right).$$
\end{proof}

\begin{proof}[Proof of Theorem \ref{th4}]
	By Theorem \ref{th3}, it is known that
	$$\sqrt{n}\left(\sum_{\ell=1}^{m}\hat{\omega}_{\ell}\hat{\alpha}_{n_{\ell},j}^{(\ell)}-\alpha_k\right)\stackrel{\mathscr{F}}{\longrightarrow} \mathcal{N}\left(0, \frac{n }{\sum_{\ell=1}^{m}n_{\ell}/\sigma_{\ell}^2}\right).$$
	For convenience, we define the random variable $Y_n=\sqrt{n}\left(\sum_{\ell=1}^{m}\hat{\omega}_{\ell}\hat{\alpha}_{n_{\ell},j}^{(\ell)}-\alpha_k\right)$ and $Y\sim \mathcal{N}\left(0, \frac{n }{\sum_{\ell=1}^{m}n_{\ell}/\sigma_{\ell}^2}\right)$, so $Y_n\stackrel{\mathscr{F}}{\longrightarrow} Y $.
	
	As $Y$ is a Gaussian distributed random variable, and it is easy to verify that $Y$ is a sub-Gaussian random variable and $\Vert Y\Vert_{\psi_2} \le C\frac{\sqrt{n} }{\sqrt{\sum_{\ell=1}^{m}n_{\ell}/\sigma_{\ell}^2}}$, then we have that
	\begin{align*}
		\Vert\sum_{\ell=1}^{m}\hat{\omega}_{\ell}\hat{\alpha}_{n_{\ell},j}^{(\ell)}-\alpha_k\Vert_{\psi_2}=&\sup_{r\ge1}r^{-1/2}\left(\mathbb{E}\vert \sum_{\ell=1}^{m}\hat{\omega}_{\ell}\hat{\alpha}_{n_{\ell},j}^{(\ell)}-\alpha_k \vert^r\right)^{1/r}\\
		=&\frac{1}{\sqrt{n}}\sup_{r\ge1}r^{-1/2}\left(\mathbb{E}\vert \sqrt{n}\left(\sum_{\ell=1}^{m}\hat{\omega}_{\ell}\hat{\alpha}_{n_{\ell},j}^{(\ell)}-\alpha_k\right) \vert^r\right)^{1/r}\\
		=&\frac{1}{\sqrt{n}}\sup_{r\ge1}r^{-1/2}\left(\mathbb{E}\vert Y_n \vert^r\right)^{1/r}.	\end{align*}
	According to Helly-Bray theorem, $\mathbb{E}\vert Y_n\vert^r\to\mathbb{E}\vert Y\vert^r $, as $n \to\infty$. Then we can easily get $\left(\mathbb{E}\vert Y_n\vert^r\right)^{\frac{1}{r}}/\sqrt{rn}\to\left(\mathbb{E}\vert Y\vert^r\right)^{\frac{1}{r}}/\sqrt{rn} $ for fixed $r$. 

		If the sequence 
		$\lim_{n\to\infty}b_n=b$. For $\forall \epsilon>0$, $\exists N$, so that, when $n>N$, has $\vert b_n-b\vert<\epsilon$. In particular, we take $\epsilon_0=1$, then
		$$\vert b_n\vert\le\vert b_n-b\vert+\vert b\vert<\vert b\vert+1,  \quad\forall n>N. $$

	In the same way, $\left(\mathbb{E}\vert Y_n\vert^r\right)^{\frac{1}{r}}/\sqrt{rn}$ is also a convergent sequence, then when $n\to \infty$, 
	\begin{align*}
		&\Vert\sum_{\ell=1}^{m}\hat{\omega}_{\ell}\hat{\alpha}_{n_{\ell},j}^{(\ell)}-\alpha_k\Vert_{\psi_2}=\sup_{r\ge1}r^{-1/2}\left(\mathbb{E}\vert \sum_{\ell=1}^{m}\hat{\omega}_{\ell}\hat{\alpha}_{n_{\ell},j}^{(\ell)}-\alpha_k \vert^r\right)^{1/r}\\
		<&\frac{1}{\sqrt{n}}\sup_{r\ge1}r^{-1/2}\left(\mathbb{E}\vert Y\vert^r\right)^{1/r} +O\left(\frac{1}{\sqrt{n}}\right)
		\le \frac{C}{\sqrt{\sum_{\ell=1}^{m}n_{\ell}/\sigma_{\ell}^2}}.
	\end{align*}
	Then we complete the proof of Theorem \ref{th4}.
\end{proof}

\bibliographystyle{myjmva}
\bibliography{ref}

\begin{thebibliography}{32}
\expandafter\ifx\csname natexlab\endcsname\relax\def\natexlab#1{#1}\fi
\providecommand{\bibinfo}[2]{#2}
\ifx\xfnm\relax \def\xfnm[#1]{\unskip,\space#1}\fi
\bibitem[{Bai(2008)}]{bai2008methodologies}
\bibinfo{author}{Z.~Bai}, \bibinfo{title}{Methodologies in spectral analysis of
  large dimensional random matrices, a review}, in:
  \bibinfo{booktitle}{Advances in Statistics}, \bibinfo{publisher}{World
  Scientific}, \bibinfo{year}{2008}, pp. \bibinfo{pages}{174--240}.
\bibitem[{Bai and Ding(2012)}]{bai2012estimation}
\bibinfo{author}{Z.~Bai}, \bibinfo{author}{X.~Ding}, \bibinfo{title}{Estimation
  of spiked eigenvalues in spiked models}, \bibinfo{journal}{Random Matrices:
  Theory and Applications} \bibinfo{volume}{1} (\bibinfo{year}{2012})
  \bibinfo{pages}{1150011}.
\bibitem[{Bai and Yao(2008)}]{bai2008central}
\bibinfo{author}{Z.~Bai}, \bibinfo{author}{J.~Yao}, \bibinfo{title}{Central
  limit theorems for eigenvalues in a spiked population model}, in:
  \bibinfo{booktitle}{Annales de l'IHP Probabilit{\'e}s et statistiques},
  volume~\bibinfo{volume}{44}, pp. \bibinfo{pages}{447--474}.
\bibitem[{Bai and Silverstein(2010)}]{bai2010spectral}
\bibinfo{author}{Z.-D. Bai}, \bibinfo{author}{J.~W. Silverstein},
  \bibinfo{title}{Spectral analysis of large dimensional random matrices},
  volume~\bibinfo{volume}{20}, \bibinfo{publisher}{Springer},
  \bibinfo{year}{2010}.
\bibitem[{Baik and Silverstein(2006)}]{baik2006eigenvalues}
\bibinfo{author}{J.~Baik}, \bibinfo{author}{J.~W. Silverstein},
  \bibinfo{title}{Eigenvalues of large sample covariance matrices of spiked
  population models}, \bibinfo{journal}{Journal of Multivariate Analysis}
  \bibinfo{volume}{97} (\bibinfo{year}{2006}) \bibinfo{pages}{1382--1408}.
\bibitem[{Battey et~al.(2018)Battey, Fan, Liu, Lu and
  Zhu}]{battey2018distributed}
\bibinfo{author}{H.~Battey}, \bibinfo{author}{J.~Fan},
  \bibinfo{author}{H.~Liu}, \bibinfo{author}{J.~Lu}, \bibinfo{author}{Z.~Zhu},
  \bibinfo{title}{Distributed testing and estimation under sparse high
  dimensional models}, \bibinfo{journal}{Annals of statistics}
  \bibinfo{volume}{46} (\bibinfo{year}{2018}) \bibinfo{pages}{1352--1382}.
\bibitem[{Cai et~al.(2020)Cai, Han and Pan}]{cai2020limiting}
\bibinfo{author}{T.~Cai}, \bibinfo{author}{X.~Han}, \bibinfo{author}{G.~Pan},
  \bibinfo{title}{Limiting laws for divergent spiked eigenvalues and largest
  nonspiked eigenvalue of sample covariance matrices}, \bibinfo{journal}{Annals
  of Statistics} \bibinfo{volume}{48} (\bibinfo{year}{2020})
  \bibinfo{pages}{1255–1280}.
\bibitem[{Chen and Peng(2021)}]{chen2021distributed}
\bibinfo{author}{S.~Chen}, \bibinfo{author}{L.~Peng},
  \bibinfo{title}{Distributed statistical inference for massive data},
  \bibinfo{journal}{Annals of Statistics} \bibinfo{volume}{49}
  (\bibinfo{year}{2021}) \bibinfo{pages}{2851--2869}.
\bibitem[{Chen et~al.(2020)Chen, Liu, Mao and Yang}]{chen2020distributed}
\bibinfo{author}{X.~Chen}, \bibinfo{author}{W.~Liu}, \bibinfo{author}{X.~Mao},
  \bibinfo{author}{Z.~Yang}, \bibinfo{title}{Distributed high-dimensional
  regression under a quantile loss function}, \bibinfo{journal}{The Journal of
  Machine Learning Research} \bibinfo{volume}{21} (\bibinfo{year}{2020})
  \bibinfo{pages}{7432--7474}.
\bibitem[{Cinar and Koklu(2019)}]{cinar2019classification}
\bibinfo{author}{I.~Cinar}, \bibinfo{author}{M.~Koklu},
  \bibinfo{title}{Classification of rice varieties using artificial
  intelligence methods}, \bibinfo{journal}{International Journal of Intelligent
  Systems and Applications in Engineering} \bibinfo{volume}{7}
  (\bibinfo{year}{2019}) \bibinfo{pages}{188--194}.
\bibitem[{Cinar and Koklu(2021)}]{cinar2021determination}
\bibinfo{author}{I.~Cinar}, \bibinfo{author}{M.~Koklu},
  \bibinfo{title}{Determination of effective and specific physical features of
  rice varieties by computer vision in exterior quality inspection},
  \bibinfo{journal}{Selcuk Journal of Agriculture and Food Sciences}
  \bibinfo{volume}{35} (\bibinfo{year}{2021}) \bibinfo{pages}{229--243}.
\bibitem[{Cinar and Koklu(2022)}]{cinar2022identification}
\bibinfo{author}{I.~Cinar}, \bibinfo{author}{M.~Koklu},
  \bibinfo{title}{Identification of rice varieties using machine learning
  algorithms}, \bibinfo{journal}{Journal of Agricultural Sciences}
  \bibinfo{volume}{28} (\bibinfo{year}{2022}) \bibinfo{pages}{307--325}.
\bibitem[{Dobriban and Sheng(2020)}]{dobriban2020wonder}
\bibinfo{author}{E.~Dobriban}, \bibinfo{author}{Y.~Sheng},
  \bibinfo{title}{Wonder: weighted one-shot distributed ridgeregression in high
  dimensions}, \bibinfo{journal}{The Journal of Machine Learning Research}
  \bibinfo{volume}{21} (\bibinfo{year}{2020}) \bibinfo{pages}{2483--2534}.
\bibitem[{Dobriban and Sheng(2021)}]{dobriban2021distributed}
\bibinfo{author}{E.~Dobriban}, \bibinfo{author}{Y.~Sheng},
  \bibinfo{title}{Distributed linear regression by averaging},
  \bibinfo{journal}{Annals of Statistics} \bibinfo{volume}{49}
  (\bibinfo{year}{2021}) \bibinfo{pages}{918--943}.
\bibitem[{Duan et~al.(2022)Duan, Ning and Chen}]{duan2022heterogeneity}
\bibinfo{author}{R.~Duan}, \bibinfo{author}{Y.~Ning},
  \bibinfo{author}{Y.~Chen}, \bibinfo{title}{Heterogeneity-aware and
  communication-efficient distributed statistical inference},
  \bibinfo{journal}{Biometrika} \bibinfo{volume}{109} (\bibinfo{year}{2022})
  \bibinfo{pages}{67--83}.
\bibitem[{Fan et~al.(2019)Fan, Wang, Wang and Zhu}]{fan2019distributed}
\bibinfo{author}{J.~Fan}, \bibinfo{author}{D.~Wang}, \bibinfo{author}{K.~Wang},
  \bibinfo{author}{Z.~Zhu}, \bibinfo{title}{Distributed estimation of principal
  eigenspaces}, \bibinfo{journal}{Annals of Statistics} \bibinfo{volume}{47}
  (\bibinfo{year}{2019}) \bibinfo{pages}{3009--3031}.
\bibitem[{Geman(1980)}]{geman1980limit}
\bibinfo{author}{S.~Geman}, \bibinfo{title}{A limit theorem for the norm of
  random matrices}, \bibinfo{journal}{The Annals of Probability}
  \bibinfo{volume}{8} (\bibinfo{year}{1980}) \bibinfo{pages}{252--261}.
\bibitem[{Gu and Chen(2022)}]{gu2022weighted}
\bibinfo{author}{J.~Gu}, \bibinfo{author}{S.~Chen}, \bibinfo{title}{Weighted
  distributed estimation under heterogeneity}, \bibinfo{journal}{arXiv preprint
  arXiv:2209.06482}  (\bibinfo{year}{2022}).
\bibitem[{Hou et~al.(2023)Hou, Zhang, Bai and Hu}]{hou2023spiked}
\bibinfo{author}{Z.~Hou}, \bibinfo{author}{X.~Zhang}, \bibinfo{author}{Z.~Bai},
  \bibinfo{author}{J.~Hu}, \bibinfo{title}{Spiked eigenvalues of noncentral
  fisher matrix with applications}, \bibinfo{journal}{Bernoulli}
  \bibinfo{volume}{29} (\bibinfo{year}{2023}) \bibinfo{pages}{3171--3197}.
\bibitem[{Huang et~al.(2023)Huang, Liu and Peng}]{huang2023distributed}
\bibinfo{author}{B.~Huang}, \bibinfo{author}{Y.~Liu},
  \bibinfo{author}{L.~Peng}, \bibinfo{title}{Distributed inference for
  two-sample u-statistics in massive data analysis},
  \bibinfo{journal}{Scandinavian Journal of Statistics} \bibinfo{volume}{50}
  (\bibinfo{year}{2023}) \bibinfo{pages}{1090--1115}.
\bibitem[{Imtiaz and Sarwate(2018)}]{imtiaz2018differentially}
\bibinfo{author}{H.~Imtiaz}, \bibinfo{author}{A.~D. Sarwate},
  \bibinfo{title}{Differentially private distributed principal component
  analysis}, in: \bibinfo{booktitle}{2018 IEEE International Conference on
  Acoustics, Speech and Signal Processing (ICASSP)},
  \bibinfo{organization}{IEEE}, pp. \bibinfo{pages}{2206--2210}.
\bibitem[{Jiang et~al.(2021)Jiang, Hou and Hu}]{jiang2021limits}
\bibinfo{author}{D.~Jiang}, \bibinfo{author}{Z.~Hou}, \bibinfo{author}{J.~Hu},
  \bibinfo{title}{The limits of the sample spiked eigenvalues for a
  high-dimensional generalized fisher matrix and its applications},
  \bibinfo{journal}{Journal of Statistical Planning and Inference}
  \bibinfo{volume}{215} (\bibinfo{year}{2021}) \bibinfo{pages}{208--217}.
\bibitem[{Johnstone(2001)}]{johnstone2001distribution}
\bibinfo{author}{I.~M. Johnstone}, \bibinfo{title}{On the distribution of the
  largest eigenvalue in principal components analysis},
  \bibinfo{journal}{Annals of Statistics} \bibinfo{volume}{29}
  (\bibinfo{year}{2001}) \bibinfo{pages}{295--327}.
\bibitem[{Koklu et~al.(2021)Koklu, Cinar and
  Taspinar}]{koklu2021classification}
\bibinfo{author}{M.~Koklu}, \bibinfo{author}{I.~Cinar}, \bibinfo{author}{Y.~S.
  Taspinar}, \bibinfo{title}{Classification of rice varieties with deep
  learning methods}, \bibinfo{journal}{Computers and Electronics in
  Agriculture} \bibinfo{volume}{187} (\bibinfo{year}{2021})
  \bibinfo{pages}{106285}.
\bibitem[{Li et~al.(2021)Li, Bao and Zhang}]{li2021robust}
\bibinfo{author}{K.~Li}, \bibinfo{author}{H.~Bao}, \bibinfo{author}{L.~Zhang},
  \bibinfo{title}{Robust covariance estimation for distributed principal
  component analysis}, \bibinfo{journal}{Metrika}  (\bibinfo{year}{2021})
  \bibinfo{pages}{1--26}.
\bibitem[{Mar{\v{c}}enko and Pastur(1967)}]{marvcenko1967distribution}
\bibinfo{author}{V.~A. Mar{\v{c}}enko}, \bibinfo{author}{L.~A. Pastur},
  \bibinfo{title}{Distribution of eigenvalues for some sets of random
  matrices}, \bibinfo{journal}{Mathematics of the USSR-Sbornik}
  \bibinfo{volume}{72(114)} (\bibinfo{year}{1967}) \bibinfo{pages}{507--536}.
\bibitem[{Paul(2007)}]{paul2007asymptotics}
\bibinfo{author}{D.~Paul}, \bibinfo{title}{Asymptotics of sample eigenstructure
  for a large dimensional spiked covariance model},
  \bibinfo{journal}{Statistica Sinica}  (\bibinfo{year}{2007})
  \bibinfo{pages}{1617--1642}.
\bibitem[{Vershynin(2010)}]{vershynin2010introduction}
\bibinfo{author}{R.~Vershynin}, \bibinfo{title}{Introduction to the
  non-asymptotic analysis of random matrices}, \bibinfo{journal}{arXiv preprint
  arXiv:1011.3027}  (\bibinfo{year}{2010}).
\bibitem[{Wang et~al.(2019)Wang, Yang, Chen and Liu}]{wang2019distributed}
\bibinfo{author}{X.~Wang}, \bibinfo{author}{Z.~Yang},
  \bibinfo{author}{X.~Chen}, \bibinfo{author}{W.~Liu},
  \bibinfo{title}{Distributed inference for linear support vector machine},
  \bibinfo{journal}{Journal of machine learning research} \bibinfo{volume}{20}
  (\bibinfo{year}{2019}) \bibinfo{pages}{1--41}.
\bibitem[{Yin et~al.(1988)Yin, Bai and Krishnaiah}]{yin1988limit}
\bibinfo{author}{Y.~Yin}, \bibinfo{author}{Z.~Bai}, \bibinfo{author}{P.~R.
  Krishnaiah}, \bibinfo{title}{On the limit of the largest eigenvalue of the
  large dimensional sample covariance matrix}, \bibinfo{journal}{Probability
  theory and related fields} \bibinfo{volume}{78} (\bibinfo{year}{1988})
  \bibinfo{pages}{509--521}.
\bibitem[{Zhang et~al.(2019)Zhang, Hu and Bai}]{Zhangh19I}
\bibinfo{author}{Q.~Zhang}, \bibinfo{author}{J.~Hu}, \bibinfo{author}{Z.~Bai},
  \bibinfo{title}{Invariant test based on the modified correction to {{LRT}}
  for the equality of two high-dimensional covariance matrices},
  \bibinfo{journal}{Electronic Journal of Statistics} \bibinfo{volume}{13}
  (\bibinfo{year}{2019}) \bibinfo{pages}{850--881}.
\bibitem[{Zhang et~al.(2022)Zhang, Zheng, Pan and Zhong}]{zhang2022asymptotic}
\bibinfo{author}{Z.~Zhang}, \bibinfo{author}{S.~Zheng},
  \bibinfo{author}{G.~Pan}, \bibinfo{author}{P.-S. Zhong},
  \bibinfo{title}{Asymptotic independence of spiked eigenvalues and linear
  spectral statistics for large sample covariance matrices},
  \bibinfo{journal}{Annals of Statistics} \bibinfo{volume}{50}
  (\bibinfo{year}{2022}) \bibinfo{pages}{2205--2230}.

\end{thebibliography}


\end{document}